\documentclass{amsart}

\usepackage{amsmath}
\usepackage{amssymb}
\usepackage{amsthm}
\usepackage{amsfonts}
\usepackage{dsfont}
\usepackage{color}
\usepackage{enumerate}
\usepackage[margin=38mm]{geometry}
\usepackage{graphicx}

\usepackage{tensor}

\usepackage{braket}

\usepackage{hyperref}

\usepackage{enumerate}  

\usepackage{url}
\usepackage{hyperref}
\usepackage[hyphenbreaks]{breakurl}

\newcommand{\reals}{\mathbb{R}}

\newcommand{\naturals}{\mathbb{N}}
\newcommand{\integers}{\mathbb{Z}}

\newcommand{\para}[1]{\left(#1\right)}
\newcommand{\paraa}[1]{\big(#1\big)}
\newcommand{\parab}[1]{\Big(#1\Big)}

\newtheorem{theorem}{Theorem}[section]
\newtheorem{corollary}[theorem]{Corollary}
\newtheorem{lemma}[theorem]{Lemma}
\newtheorem{proposition}[theorem]{Proposition}
\newtheorem{example}[theorem]{Example}
\theoremstyle{definition}
\newtheorem{definition}[theorem]{Definition}
\theoremstyle{remark}
\newtheorem{remark}[theorem]{Remark}
\numberwithin{equation}{section}

\title{Fuzzy cylinders of finite length}

\author{Andreas Sykora}
\address[Andreas Sykora]{}
\email{syko@gelbes-sofa.de}

\thanks{}

\subjclass[2000]{}

\begin{document}

\begin{abstract}
   We introduce non-commutative algebras, which can be associated with the function
   algebra of functions on a finite or half-finite cylinder. The algebras,
   which depend on a deformation parameter, are crossed product algebras
   of a partial group action of $\integers$ on an interval of the real line $\reals$. 
   Discrete representations of the algebras can be seen as matrix regularizations
   of the respective function algebra on the finite or half-finite cylinder
   and therefore as fuzzy space.
   
   In a second part of the article, we review crossed product algebras
   based on partial group actions and derive the results needed
   in the first part.
\end{abstract}

\maketitle

\tableofcontents

\section{Introduction}

The commuting Gelfand-Naimark theorem states that
every commuting $C^*$-algebra is isomorphic to an algebra
of continuous functions on a topological space. Generalizing this,
one may assume that every non-commuting $C^*$-algebra
is an algebra of functions on a noncommutative topological space.
These considerations play an important role in noncommutative
geometry, \cite{Conn_1994}, \cite{Mad_1999}.

Of particular interest are $C^*$-algebras, which depend
on a parameter $\hbar$, which become commuting, when this
parameter approaches 0. Such algebras may be seen as a fuzzyfication
of the topological space associated with the commuting algebra at $\hbar=0$,
which in this way can be seen as classical commuting limit of the noncommutative space.

In \cite{Arnlind_2022} shift algebras were introduced,
which are based on the action of the
integers $\integers$ on the real line $\reals$ by a shift $S_{\hbar,n}:x\mapsto x+n\hbar$.
Here, $n\in \integers$ and $\hbar$ is a positive real constant.
The action $S_{\hbar,n}$ of $\integers$ on the real line induces an
action $\hat{S}_{\hbar,n}$ on the function algebra
$\mathcal{F}(\reals)$ of complex valued functions on
$\reals$ by pullback, i.e. $f(x)\mapsto f(S_{\hbar,n}(x))=f(x+n\hbar)$.
The corresponding shift algebra is the crossed product algebra
of $\integers$ with $\mathcal{F}(\reals)$ by this induced action. It was
then shown that with the aid of discrete representations of
special subalgebras of the shift algebra, a lot of fuzzy spaces can be reproduced,
i.e. matrix algebras, which in a sense become a continuous
manifold in the classical limit $\hbar\rightarrow 0$.

The shift action $S_{\hbar,n}=S_{\hbar,1}^n$ is based on
the constant shift $S_{\hbar,1}$, which is a bijection of $\reals$
onto itself. A natural generalization of this are partial bijection,
which are bijections between subsets of a set. In this
context, the group action is replaced by a partial action. Such actions
arise by restricting group actions to subsets. In general, the global bijections
of the set are replaced by bijections between subsets, i.e. partial bijections.
For example, we can restrict the action $S_{\hbar,n}$ of $\integers$
on $\reals$ to a subset $I\subset\reals$. We set
 $I_n=I \cap S_{\hbar,n}(I)$ and restrict the domain of
 $S_{\hbar,n}$ to $I_{-n}$, which then results in partial
 bijections $S_{\hbar,n}:I_{-n}\rightarrow I_{n}$. This
 results in a partial action of $\integers$ on the subset $I$.

In \cite{Sieben_1996} and \cite{Exel_1998} partial group actions and 
inverse semigroup actions, which naturally arise in the context
of partial group actions, have been considered in the context of
$C^*$-algebras and crossed product $C^*$-algebras
were constructed based on these actions. Special properties of
$C^*$-algebras, such as approximate identities have been used
to show that the crossed product $C^*$-algebras becomes associative
and that their representations are covariant representations.

The present article is divided into two parts, where in 
the second part we review crossed product algebras based
on semigroup actions and partial group actions, such as 
introduced in \cite{Sieben_1996}. Contrary to \cite{Sieben_1996}
and \cite{Exel_1998}, we will however stay in the context of
ordinary $*$-algebras, to arrive at associative crossed product algebras
with covariant representations. The second part is mainly a reference for
the definitions and theorems, we are using in the first part. 

In the first part, we introduce noncommutative spaces,
which are based on the crossed product algebra of the partial group action
of $\integers$ on the complex valued functions on
an interval $I\subset\reals$, wherein the partial group action
is based on a single partial bijection of $I$. As examples,
we construct a fuzzy cylinder of finite length and try to construct a
fuzzy Poincare disc. 

\part{Fuzzy cylinders}

In the present article, a fuzzy space is seen as a family of
noncommutative algebras that in a sense have a so called "classical limit",
i.e. the family depends on a deformation parameter, such
as in deformation quantization, \cite{Klim_1992}, \cite{Bord_1994}.

We assume that the algebra with zero deformation parameter
becomes commuting. As already mentioned, the commuting 
Gelfand-Naimark theorem then states that a commuting $C^*$-algebra
is isomorphic to $C_0(X)$, the continuous functions on a compact
Hausdorff space $X$, which vanish at infinity. In such a way,
the commutative algebra can be associated with the space $X$,
which becomes "fuzzy", when the deformation parameter $\hbar$
is different from $0$.

This part is organized as follows: In section \ref{sec:alp_fuz}
we introduce a specific fuzzy space, the $\alpha$-fuzzy cylinder,
which is a crossed product algebra of a partial group action of
$\integers$ on an interval $I$ of the real line $\reals$. The 
following section \ref{sec:alp_fuz_prop} is devoted to basic properties
of general $\alpha$-fuzzy cylinder. In section
\ref{sec:shift_cyl} we consider fuzzy cylinders based on a
constant right shift applied to finite intervals and are able to
define a fuzzy cylinder of finite length, which is seen as the main 
result of this work. In section \ref{sec:cl_lim} 
it is shown that the classical limit of the $\alpha$-fuzzy cylinder is really
a topological cylinder. In section \ref{sec:spec_sub} we
consider subalgebras of the $\alpha$-fuzzy cylinder and
try to construct an algebra representing a fuzzy Poincare disc.

\section{$\alpha$-fuzzy cylinders} \label{sec:alp_fuz}

All the fuzzy spaces introduced in the following are
based on a family of partial bijection $\alpha_\hbar$
of an interval $I\subset\reals$, which we will use to
construct a crossed product algebra. The interval can
have open and/or closed ends and also can
go to $-\infty$ and $+\infty$. Recall that a partial
bijection of a set is a bijection between subsets of the set,
see section \ref{sec:inv_semi}.
The deformation parameter $\hbar$ is from a set
of non-negative real parameters containing $0$.
We demand that $\alpha_\hbar$ is continuous
for every $\hbar$ and that $\alpha_0=\text{id}_I$.
Thus, for $\hbar=0$, the partial bijection becomes the
identity of $I$, $\alpha_0=\text{id}_I$. 
We further demand that $\alpha_\hbar \rightarrow \alpha_0=\text{id}_I$ for
$\hbar\rightarrow 0$ pointwise, such that for
decreasing $\hbar$ the partial bijections becomes more and
more the identity mapping.
\begin{definition}
  We say that such a partial bijection $\alpha_\hbar$ has
  \emph{classical limit}.
\end{definition}

When we concatenate $\alpha_\hbar$ with itself, the result
is a further partial bijection, however possibly with another
domain and range. Proposition \ref{single_bij_action} results in:

\begin{proposition} \label{int_pga}
The set $\{\alpha_\hbar^n, n\in\integers \}$ of partial bijections generated
by the partial bijection $\alpha_\hbar$ of the interval $I$ with classical limit
is a partial group action of $\integers$ on 
the interval $I$. With $\alpha_\hbar:I_{\hbar,-1}\rightarrow I_{\hbar,1}$,
where $I_{\hbar,-1}\subset I$ and $I_{\hbar,1} \subset I$
are the subintervals of the interval $I\in\reals$, which are
mapped to each other by $\alpha_\hbar$, the domain and ranges $I_{n}$
of $\alpha_\hbar^n:I_{-1}\rightarrow I_n$ (by
omitting $\hbar$ as index) are
\begin{align*}
      I_{n} =\alpha^{n-1} \paraa{I_\cap} \cap \cdots \cap \alpha \paraa{I_\cap}
\end{align*}
for $n>1$ and $n<1$ where $I_\cap=I_{-1}\cap I_1$.

For $\hbar=0$, the partial group action becomes the
trivial group action of $\integers$ on $I$, in which case
$\alpha_0^n=\text{id}_I$ and $I_n=I$.
\end{proposition}

As shown in proposition \ref{pga_semigroup}, every
partial group action gives rise to an inverse semigroup.
The inverse semigroup $S_{\alpha_\hbar}(\integers)$
generated from a single partial bijection is further
described in proposition \ref{single_bij_sg}, where
it is shown that it elements are composed of projectors 
$\alpha_\hbar^n\circ \alpha_\hbar^{-n}$ on the
sets $I_n$ and the mappings $\alpha_\hbar^m$.

According to proposition \ref{fun_alg_act}, a
semigroup action on a set, here the interval $I$,
induces an inverse semigroup action on the function algebra
on the set, i.e. the function algebra $\mathcal{F}(I)$.
When $f$ is a function on $I$, the mapping induced by $\alpha_\hbar$
is simply defined by concatenation, i.e. 
$\hat\alpha_\hbar(f)=f\circ \alpha_\hbar^{-1}$, where
$-1$ indicates the inverse partial bijection. Since however
$\alpha_\hbar$ is solely a partial mapping, this is not possible
for all functions $f$ but only for those, who vanish outside the
domain of $\alpha_\hbar$, which form an ideal of the 
algebra of functions on $I$. 

As we have demanded that $\alpha_\hbar$ is a continuous
mapping, we receive also an inverse semigroup action
on $C(I)$, the algebra of continuous functions on the
interval $I$, see remark \ref{com_fun_alg}.

In summary, starting with a single partial bijection
$\alpha_\hbar$ of an interval $I$, we arrive at an
inverse semigroup action of $\integers$ on the
function algebra $\mathcal{F}(I)$ (or $C(I)$). With such
data, it is possible to construct a crossed product
algebra, such as defined in definition \ref{def_pi_cpa}
(or definition \ref{def_alg_cpa}). These crossed product
algebras will be the basis of the fuzzy spaces described
herein.

\begin{definition} \label{alp_fuz_cyl}
The \emph{$\alpha$-fuzzy cylinder}
based on the partial bijection $\alpha_\hbar$ of the interval $I$
with classical limit is the $*$-algebra
$\text{Cyl}(I, \alpha_\hbar)=\mathcal{\tilde{F}}(I) \rtimes_{\alpha_\hbar} \integers$
(see proposition \ref{single_bij_cpa}), where
$\mathcal{\tilde{F}}(I)$ is one of the algebras 
$\mathcal{F}(I)$, $C(I)$ or $C_0(I)$, i.e. 
the algebra of all functions, the algebra of
continuous functions or the algebra of continuous
functions vanishing at infinity on $I$.
\end{definition}

The algebra properties of the $\alpha$-fuzzy cylinder
follow from proposition \ref{single_bij_cpa}.
\begin{proposition} \label{alp_fuz_cyl_alg}
The elements of $\text{Cyl}(I, \alpha_\hbar)$ are
mappings from $\integers$ into $\mathcal{\tilde{F}}(I)$
of the form
\begin{align*} 
    \phi = \sum_{n\in \integers}^{\text{finite}} f_n \delta_n
\end{align*}
where $f_n \delta_n$ is the mapping, which maps $n\in\integers$ to
the function $f_n$, which is in the ideal $\mathcal{\tilde{F}}(I_n)$,
i.e. which vanishes outside $I_n$. 
The algebra multiplication and involution are
\begin{align*} 
   f_n \delta_n \cdot  f_m \delta_m = f_n \paraa{ f_m |_{I_{-n}} \circ \alpha_\hbar^{-n} } \delta_{n+m} 
\end{align*}
\begin{align*} 
   (f_n\delta_n)^* = \paraa{\overline{f_n}\circ \alpha^{n}}\delta_{-n}
\end{align*}
where $f_n\in \mathcal{\tilde{F}}(I_n)$ and $f_m\in \mathcal{\tilde{F}}(I_m)$.
\end{proposition}
Note that the symbol $\delta_n$ alone makes no sense, since the unit function $f=1$ is not part
of the ideals $\mathcal{\tilde{F}}(I)$ except for $n=0$. In the case of
the function algebra being $C_0(I)$ it is even possible to define a
crossed product $C^*$-algebra and infinite sums are possible.

The expression $f_m |_{I_{-n}}$ means that the function $f_m$, which
solely has support $I_m$, is further restricted to $I_{-n}$, such that
$\alpha^{n}$, which has range $I_{-n}$, can be applied. The result
is multiplied with a function in $\mathcal{\tilde{F}}(I_n)$, which ensures that
the product is again a function in $\mathcal{\tilde{F}}(I)$.
In the second part it is shown that this really result in a well-defined associative product.

In the case that $\mathcal{\tilde{F}}(I)$ is the
algebra of all functions $\mathcal{F}(I)$, the ideals
$\mathcal{F}(I_n)$ contain the partial identities $p_n$
(definition \ref{pi_def}), which are the functions, which
are $1$ on $I_n$ and $0$ on $I\setminus I_n$.
Proposition \ref{single_bij_cpa_pi}, which deals with
this type of crossed product algebras, then tells:
\begin{proposition} \label{alp_fuz_cyl_pi}
The elements of $\text{Cyl}(I, \alpha_\hbar)$ 
based on the function algebra of all functions
 $\mathcal{F}(I)$ have the form
\begin{align*} 
   \phi =\sum_{n\in \naturals}^{\text{finite}} f_{-n} U^{*n} 
           + f_0 + \sum_{n\in \naturals}^{\text{finite}} f_n U^n
\end{align*}  
where $f_n\in \mathcal{F}(I)$,
$U=p_1\delta_1$ and $U^*=p_{-1}\delta_{-1}$. 
The elements $U$ and $U^*$ fulfill 
\begin{align*}   
    U U^* = p_1=p_{I_1}, \qquad
    U^* U = p_{-1} = p_{I_{-1}}, \\
    U f = \paraa{(p_{-1} f)\circ \alpha_\hbar^{-1}} U, \qquad
    U^* f = \paraa{(p_{1} f)\circ \alpha_\hbar} U^*
\end{align*}  
for $f\in \mathcal{F}(X)$.
\end{proposition}
Note that $UU^*U=U$ and
$U^*UU^*=U^*$ and the elements $0$, $1$ and $U$
generate an inverse semigroup.

With the aid of proposition \ref{point_hs_rep_pg} representations
of the $\alpha$-fuzzy cylinder $\text{Cyl}(I, \alpha_\hbar)$
are easily determined.
\begin{proposition} \label{alp_fuz_cyl_rep}
  Let $\text{Cyl}(I, \alpha_\hbar)$ be a
  $\alpha$-fuzzy cylinder and let $x_0$ be
  a point in the interval $I$. Then
  there are integers $n_- \leq 0$ (which can be $-\infty$)
  and $n_+ \geq 0$ (which can be $+\infty$), such that
 \begin{align*}  
   \tilde{X} = \{ \alpha_\hbar^n(x_0) | n\in \integers, x_0\in I_{-n} \}
      = \{ x_n| n_- \leq n \leq n_+ \}
 \end{align*}  
  Let $\mathcal{H}$  be the Hilbert space with 
  basis $\{\ket{x_n} | x_N\in\tilde{X}, n_- \leq n \leq n_+ \}$ and 
  let
  \begin{align*}
   V \ket{x_n} = \begin{cases}
       	              \ket{x_{n+1} }, \text{ if } n_- \leq n < n_+  \\
                           0, \text{ if } n = n_+
                        \end{cases} & ,  \qquad
   V^* \ket{x_n}  = \begin{cases}
       	              \ket{x_{n-1} }, \text{ if } n_- < n \leq n_+  \\
                           0, \text{ if } n = n_-
                        \end{cases} ,  \\
   \pi(f) \ket{x_n} & = f(x_n) \ket{x_n}
\end{align*}
Then
\begin{align*} 
   \sum_{n\in \integers}^\text{finite} f_n \delta_n \mapsto
   \sum_{n\in \naturals}^{\text{finite}} \pi\paraa{f_{-n}} (V^*)^{n} 
         + \pi(f_0) + \sum_{n\in\naturals}^{\text{finite}} \pi\paraa{f_n} V^n
\end{align*}  
is a $*$-representation of $\text{Cyl}(I, \alpha_\hbar)$.
\end{proposition}
\begin{proof}
Let
\begin{align*}  
   \tilde{X}=\tilde{X}(\alpha_\hbar, x_0)
      = \{ \alpha_\hbar^n(x_0) | n\in \integers, x_0\in I_{-n} \}
\end{align*}  
be the set of definition \ref{point_hs} and proposition \ref{point_hs_pga}.
Since the domains and ranges $I_n$ of the $\alpha_\hbar^n$ are
contained within each other $I_{n+1} \subset I_n$ and
$I_{-n-1} \subset I_{-n}$ for $n\geq 1$, there is possibly a minimal $n_-\leq 0$ and
a maximal $n_+ \geq 0$, for which $x_n=\alpha_\hbar^n(x_0)$ is defined 
for the last time. 
Otherwise $n_-$ and/or $n_+$ are $-\infty$ and $+\infty$,
respectively.
Thus $\tilde{X} = \{ x_n| n_- \leq n \leq n_+ \}$.
Depending on $\alpha_\hbar$, this set is finite or infinite.

It follows that the Hilbert space of proposition \ref{point_hs_rep_pg}
has basis $\{\ket{x_n} | n_- \leq n \leq n_+ \}$.
The covariant representation of proposition \ref{point_hs_rep_pg}
is defined by
\begin{align*}
   V_n \ket{x_m} = \begin{cases}
       	              \ket{\alpha_\hbar^n(x_m)}, \text{ if } x_m \in I_{-n} \\
                           0, \text{ otherwise}
                        \end{cases} ,  \qquad
      \pi(f) \ket{x_n} = f(x_n) \ket{x_n}
\end{align*}
When we define $V$ as above, it follows that $V_n=V^n$ 
and $V_{-n}=(V^*)^n$ for $n>0$. 
Thus, according to corollary \ref{point_hs_cpa_rep}, 
the above mapping is a $*$-representation of $\text{Cyl}(I, \alpha_\hbar)$.
\end{proof}

We have arrived at a matrix representation with matrices of size $n_+-n_-$,
in which the functions are mapped to diagonal matrices and the partial
bijection $\alpha_\hbar$ is represented by a shift matrix $V$.

When $\mathcal{\tilde{F}}(I)$ is
the algebra of all functions $\mathcal{F}(I)$,
the algebra element $U$ of proposition \ref{alp_fuz_cyl_pi} is mapped to $V$.

\section{Properties of $\alpha$-fuzzy cylinders} \label{sec:alp_fuz_prop}

Here, we collect some results for
general $\alpha$-fuzzy cylinders. We first note that we also can use
the inverse partial bijection of $\alpha_\hbar$.
\begin{proposition} \label{inv_alpha}
$\text{Cyl}(\alpha_\hbar)=\text{Cyl}(\alpha_\hbar^{-1})$ 
\end{proposition}
\begin{proof}
  The inverse semigroup of $\mathcal{\tilde{F}}(I) \rtimes_{\alpha_\hbar} \integers$
  is generated by $\alpha_\hbar^n$ for $n\in \integers$. This set is also generated
  by $\alpha_\hbar^{-1}$.
\end{proof}

From a topological point of view,
we can restrict to $\alpha$-fuzzy cylinders on the
real line $\reals$, the half-line $[0,\infty[$ and the interval $[0,1]$.
\begin{proposition} \label{pb_cl_IJ}
    Let $\alpha_\hbar$ be partial bijection of the interval $I$ with classical limit
    and let $\beta:I\rightarrow J$ be a continuous bijection from the interval $J$ into the
    interval $I$. Then $\alpha'_\hbar=\beta\circ\alpha_\hbar\circ\beta^{-1}$
    is a partial bijection with classical limit of the interval $J$.
\end{proposition}
\begin{proof}
  As $\beta$ is a continuous bijection, so is $\beta^{-1}$ and $\alpha'_\hbar$
  is continuous. Furthermore $\alpha'_0=\beta\circ\alpha_0\circ\beta^{-1}=\text{id}_J$.
\end{proof}
\begin{proposition}
  Each $\alpha$-fuzzy cylinder $\text{Cyl}(I, \alpha_\hbar)$ is homeomorphic
  to an $\alpha$-fuzzy cylinder $\text{Cyl}(J, \alpha'_\hbar)$ 
  on $J=\reals$, $J=[0,\infty[$ or $J=[0,1]$.
\end{proposition}
\begin{proof}
  Depending on whether the interval is $I$ is open at both ends,
  open at one end or closed, we can find a $\beta$ such in
  proposition \ref{pb_cl_IJ} that maps to $J=\reals$, $J=[0,\infty[$ or $J=[0,1]$.
\end{proof}

Note that the domain and range of $\beta$ can be the
same and we can map fuzzy cylinders of different $\alpha_\hbar$ 
on the same interval $I$ to each other. The question arise,
whether there is solely one partial bijection $\alpha_\hbar$ 
representing all other ones. However, when $x$ is a fixed point of
$\alpha_\hbar$, i.e. $\alpha_\hbar(x)=x$, then
the same applies to the point $\beta(x)$ for
$\alpha'_\hbar=\beta\circ\alpha_\hbar\circ\beta^{-1}$.
Thus, it is not possible to reduce
the number of fixed points of a partial bijection
$\alpha_\hbar$ by applying
a continuous bijection $\beta$.
Partial bijections with classical limit with any number
of fixed points are easily constructed. This shows
that the shift based fuzzy cylinders of the
next section are not the only possible solutions.

In particular, $\alpha$-fuzzy cylinders based on
partial bijections $\alpha_\hbar$ with classical limit having fixed points
are interesting on their own, since they give rise
to $\alpha$-fuzzy cylinders, which have commuting
subalgebras. 

\begin{proposition}
  Let $\alpha_\hbar$ be a partial bijection
  of the interval $I$ with classical limit and
  let $S\subset I$ be a subset of fixed points
  of $\alpha_\hbar$, i.e. $\alpha_\hbar(x)=x$ for $x\in S$.
  Then the algebra elements
  \begin{align*} 
      \phi = \sum_{n\in \integers}^{\text{finite}} f_n \delta_n
  \end{align*}
  with $f_n\in \mathcal{\tilde{F}}(S)$
  of $\text{Cyl}(I, \alpha_\hbar)$
  form a commuting subalgebra.
\end{proposition}
\begin{proof}
  Assume that that 
  $f \in \mathcal{\tilde{F}}(S)$, i.e. that
  $f$ is a function on the set $S$, which vanish outside $S$.
  It follows that $f \in \mathcal{\tilde{F}}(I_n)$, since
 $S\subset I_n$ because $\alpha_\hbar^n(S)=S$.
  An element $\phi$ of $\text{Cyl}(I, \alpha_\hbar)$ such
  as defined above is therefore well defined.

  The algebra product of two such elements
  commutes, since $\alpha_\hbar$ is the identity on $S$,
  and is again of this form, since the product of two function
  on $S$ is again a function on $S$.
\end{proof}

For example, the subset $S$ can be a subinterval of $I$. In
this case the fuzzy cylinder $\text{Cyl}(I, \alpha_\hbar)$
has $\text{Cyl}(S, \text{id}_S)$ as subalgebra, which
is the algebra of functions on the cylinder $S\times S_1$.
This may be seen as an example of a \emph{partially fuzzy space}.

\section{Shift based fuzzy cylinders} \label{sec:shift_cyl}

The basic data of the $\alpha$-fuzzy cylinder
of definition \ref{alp_fuz_cyl} is an interval and
a partial bijection of this interval. From a topological
point of view there are three different types of 
 intervals $[a,b], a<b \in \reals$, $[a,\infty[a\in \reals$ and $]\infty,\infty[=\reals$.
 In this section three different fuzzy cylinders, which are
based on these three types of intervals and a constant
right shift are described.

\begin{definition}
  For $\hbar > 0$ the \emph{right shift} is the mapping
   $S_\hbar: \reals\rightarrow\reals,x\mapsto x+\hbar$.
\end{definition}

When we restrict the right shift to an interval in $\reals$,
we arrive at a partial bijection with classical limit, which we can use for the
$\alpha$-fuzzy cylinder. 

\begin{definition} \label{sb_fuz_cyl}
   For an interval $I \subset \reals$, 
   the \emph{(shift based) fuzzy cylinder} $\text{Cyl}_\hbar(I)$
   is the $\alpha$-fuzzy cylinder $\text{Cyl}(I, S_\hbar)$, 
   where the partial bijection $\alpha_\hbar=S_\hbar$
   is the right shift $S_\hbar$ restricted
   to the subsets $I_1=S_\hbar(I)\cap I$ and 
   $I_{-1}=S^{-1}_\hbar(I)\cap I$ of $I$.
   
   For the intervals $[a,b]$, $[a,\infty[$, $]-\infty,a]$ and $\reals$,
   we call the corresponding
   fuzzy cylinder \emph{finite}, \emph{half-finite} and
   \emph{infinite}.
\end{definition}

Note that the right shift generates a group action
$S_\hbar^n$ ($n\in\integers$) of the group of integers 
$\integers$ on the set $\reals$. According to
proposition \ref{restricted_action}, a group
action on a set restricted to a subset is a partial action, which
can be used to define a crossed product algebra. This
is an alternative way to define the shift based fuzzy cylinder.

In the case of the infinite fuzzy cylinder, i.e. $I=\reals$,
the right shift $S_\hbar$ is not restricted and generates
a group action,i.e. $I_n=\reals$. Therefore, the
identity function $p=1$ is the partial identity for all $I_n$.
The corresponding algebra is the same as the
one of the fuzzy cylinder described in \cite{Sui_2004},
\cite{Stei_2011}, \cite{Arnlind_2022}, \cite{Arn_2020}.

\begin{proposition}
  The elements of the infinite fuzzy cylinder $\text{Cyl}_\hbar(\reals)$
  are of the form
  \begin{align*} 
   \phi =\sum_{n\in \integers}^{\text{finite}} f_{n} U^{n} 
   \end{align*}  
   where $f_n\in \mathcal{\tilde{F}}(I)$.
   The algebra relations follow from
   \begin{align*}   
      U^* = U^{-1}, \qquad f^*=\overline{f}, \qquad
      U f(x) = f(x-\hbar) U, \qquad
      U^* f(x) = f(x+\hbar) U^*
   \end{align*}
   Covariant representations of the infinite fuzzy cylinder are
   \begin{align*}   
      V \ket{\hbar n + \delta} = \ket{n+1)} ,  \qquad
      \pi(f) \ket{\hbar n +\delta} = f(\hbar n +\delta) \ket{\hbar n +\delta}
   \end{align*}
   for $f\in \mathcal{\tilde{F}}(I)$ and $V$, which is the operator
   representing $U$, being a unitary operator.
\end{proposition}

For the finite and half-finite cylinder, the algebra product and
algebra involution do not simplify and are described
in propositions \ref{alp_fuz_cyl_alg} and \ref{alp_fuz_cyl_pi}.
However, we can further investigate the intervals $I_n$ of
proposition \ref{int_pga} and
the representation of proposition \ref{alp_fuz_cyl_rep}.

\begin{proposition} \label{fin_sb_cyl}
  Let $\text{Cyl}_\hbar([a,b])$  with $a<b \in \reals$
  be a finite fuzzy cylinder. 
  
  Let $N$ be the smallest natural number, such that
  $N \hbar > b-a$, then the domain and
  ranges of the mappings of the partial group action of
  the restricted right shift $S_\hbar^n:I_{-n}\rightarrow I_n$ are
    \begin{align*} 
      I_n = [a+n\hbar, b], \qquad
      I_{-n} = [a, b- n \hbar]
  \end{align*}
  for $0 < n \leq N$. For $n>N$, $I_n$ and $I_{-n}$ are empty.
  Thus, the elements of $\text{Cyl}_\hbar([a,b])$ are
  \begin{align*} 
    \phi = \sum_{-N \leq n \leq N} f_n \delta_n
\end{align*}
  with the functions $f_n$ in the ideals $\mathcal{\tilde{F}}(I_n)$. 
  
  In the nomenclature of proposition \ref{alp_fuz_cyl_rep}, the
  set $\tilde{X}$ is $\{x_0 + n\hbar | 0 \leq n < N \}$ with
  $x_0 \in [a,a+\hbar[$, i.e. $n_-=0$ and $n_+=N-1$.
\end{proposition}
\begin{proof}
  The restriction of the right shift to the interval $[a,b]$ is 
  \begin{align*} 
    S_\hbar: I_{-1}=[a,b-\hbar] \rightarrow I_1= [a+\hbar,b]
  \end{align*} 
  When $b-a>\hbar$, it follows that $I_{-1}$ and $I_1$
  are empty and that all $I_n$ are empty. Otherwise
  $I_\cap=[a+\hbar,b-\hbar]$ is non-empty and
  \begin{align*} 
      I_n = S_\hbar^{n-1}( I_\cap) = [a+n\hbar, b], \qquad
      I_{-n} = S_\hbar^{-n+1}( I_\cap) = [a, b- n\hbar]
  \end{align*}
  (see proposition \ref{int_pga}).
  Thus, for $N$ the smallest natural number such that $N\hbar > b-a$,
  the intervals $I_n$ and $I_{-n}$ with $n>N$ are empty.

  For the representation, we can assume that
  $x_0=\in \reals$ is in the interval
  $[a,a+\hbar[$. Otherwise we can apply $S_\hbar^{-1}$ as long
  the result is bigger than $a+\hbar$.
  It follows that the set $\tilde{X}$ 
  is $\{n\hbar + \delta | 0 \leq n < N \}$.
\end{proof}

The resulting $*$-representation of proposition \ref{alp_fuz_cyl_rep}
is restricted to sums between $-N$ and $N$. In particular,
 the right-shift and left-shift matrices $V$ and $V^*$ have the
 property that $V^{N+1}=0$ and $(V^*)^{N+1}=0$.
 
The corresponding results for the half-finite fuzzy cylinder are:
\begin{proposition}
  Let $\text{Cyl}_\hbar([a,\infty[)$  with $a \in \reals$
  be a half-finite fuzzy cylinder. 
  
  The domain and
  ranges of the mappings of the partial group action of
  the restricted right shift $S_\hbar^n:I_{-n}\rightarrow I_n$ are
    \begin{align*} 
      I_n = [a+n\hbar, \infty[, \qquad
      I_{-n} = I= [a, \infty[
  \end{align*}
  for $n\in \naturals_0$.
  The elements of $\text{Cyl}_\hbar([a,\infty[)$ are
  \begin{align*} 
    \phi = \sum_{n \in\integers}^\text{finite} f_n \delta_n
  \end{align*}
  with the functions $f_n$ in the ideals $\mathcal{\tilde{F}}(I_n)$. 
  
  In the nomenclature of proposition \ref{alp_fuz_cyl_rep}, the
  set $\tilde{X}$ is $\{x_0 + n\hbar | n\in \naturals_0\}$ with
  $x_0 \in [a,a+\hbar[$, i.e. $n_-=0$ and $n_+=\infty$.
\end{proposition}

\begin{example}The finite fuzzy cylinder of order $N$\end{example}
\noindent As the shift based fuzzy cylinder is considered
as one of the main results of this work, we describe the
finite fuzzy cylinder $\text{Cyl}_\frac{1}{N}([0,1])$
of length $1$ with $\hbar=\frac{1}{N}, n\in\naturals$ and
based on the algebra of all functions $\mathcal{F}([0,1])$
in detail. 

The interval $I$ is the closed interval $[0,1]$ from $0$ to $1$.
For $\hbar=\frac{1}{N}$, the right shift on this interval 
becomes a partial bijection from 
\begin{align*} 
   S_\frac{1}{N}: I_{-1}=[0,1-\frac{1}{N}] \rightarrow I_1= [\frac{1}{N},1],
   x\mapsto x+\frac{1}{N}
\end{align*} 
As we are using the function algebra of all functions
$\mathcal{F}([0,1])$, we are allowed to use partial identities with
support on subintervals (see proposition \ref{single_bij_cpa_pi}).
For example the partial identity $p_1$ is the function, 
which is $1$ on $I_1$ and zero otherwise.

From propositions \ref{fin_sb_cyl} and \ref{single_bij_cpa_pi})
it follows that the elements of $\text{Cyl}_\frac{1}{N}([0,1])$ are 
\begin{align*} 
    \phi =\sum_{0< n < N} f_{-n} U^{*n} 
           + f_0 + \sum_{0<n <N} f_n U^n
\end{align*}  
 where $f_n\in \mathcal{F}([0,1])$, 
$U=p_1\delta_1$ and $U^*=p_{-1}\delta_{-1}$.
The elements $U$ and $U^*$ fulfil 
\begin{align*}   
    U^N=0, \qquad U^{*N}=0, \qquad U U^* = p_1, \qquad   U^* U = p_{-1}, \\
    U f(x) = (p_{-1}f)\paraa{ x-\frac{1}{N} } U, \qquad
    U^* f(x) = (p_{1}f) \paraa{ u+\frac{1}{N} } U^*
\end{align*}  
for $f\in \mathcal{F}([0,1])$.

Proposition \ref{alp_fuz_cyl_rep} describes a representation. In particular,
Let $\mathcal{H}$  be the $N$-dimensional Hilbert space with 
basis $\{\ket{n} | 0 \leq n < N \}$ and let
\begin{align*}
  V \ket{n} = \begin{cases}
      	              \ket{n+1}, \text{ if } 0 \leq n < N - 2  \\
                         0, \text{ if } n = N - 1
                   \end{cases} & ,  \qquad
  V^* \ket{n}  = \begin{cases}
       	              \ket{n-1}, \text{ if } 0 < n < N-2  \\
                           0, \text{ if } n = 0
                      \end{cases} ,  \\
   \pi(f) \ket{n} & = f(x_0+\frac{n}{\delta}) \ket{n}
\end{align*}
where $x_0\in[0,\frac{1}{N}[$. Then 
\begin{align*} 
   \phi \mapsto
   \sum_{n\in \naturals}^{\text{finite}} \pi\paraa{f_{-n}} (V^*)^{n} 
         + \pi(f_0) + \sum_{n\in\naturals}^{\text{finite}} \pi\paraa{f_n} V^n
\end{align*}  
 is a $*$-representation of $\text{Cyl}_\frac{1}{N}([0,1])$,
 which is the algebra of $N\times N$-matrices. The
functions on the interval $[0,1]$ are mapped to diagonal matrices, wherein
the $n$th diagonal entry is the function value at the point $x_0+\frac{n}{N}$.
The right shift and its conjugate are represented by the shift matrices $V$ 
and $V^*$. Section \ref{sec:cl_lim} below shows that for $N\rightarrow\infty$,
the fuzzy cylinder $\text{Cyl}_\frac{1}{N}([0,1])$ becomes the algebra
of functions on the cylinder $[0,1]\times S_1$.

Interestingly, the representations of the finite fuzzy cylinder
can be generated by applying a finite projector to the infinite
representations of the infinite fuzzy cylinder. The same is
done with the fuzzy plane to generate the fuzzy disc in
\cite{Lizzi_2003}. Remarkably, both are fuzzy spaces having
a border.

\section{Classical limit}  \label{sec:cl_lim}

When the $\alpha$-fuzzy cylinder $\text{Cyl}(I, \alpha_\hbar)$
is considered as a $C^*$-algebra, its classical limit $\hbar\rightarrow 0$ 
is the topological space $I\times S_1$, i.e. a cylinder.
Defining a star product, we also show that we can endow
this cylinder with a Poisson structure depending on $\alpha_\hbar$.

\subsection{$C^*$-algebra}

We note that the $\alpha$-fuzzy cylinder algebra
$\text{Cyl}(I, \alpha_\hbar)=\mathcal{\tilde{F}}(I) \rtimes_{\alpha_\hbar} \integers$
of definition \ref{alp_fuz_cyl} can be completed to
a $C^*$-algebra, when the algebra of functions $\mathcal{\tilde{F}}(I)$
is $C_0(I)$. This is shown in \cite{Sieben_1996}, to which
we largely refer in the second part of this article.
For $\hbar=0$, this $C^*$-algebra
becomes a commuting algebra and therefore has to be
a function algebra on a locally compact Hausdorff space
(Gelfand representation theorem), which we then interpret 
as the "classical limit".

In particular, we can define the $C^*$-algebra
$\text{Cyl}^*_{\alpha_\hbar}(I)= C_0(I) \hat\rtimes_{\alpha_\hbar} \integers$
such as in definition 4.1 of \cite{Exel_2010}, where we use the
partial group action of proposition \ref{single_bij_action} of
the present article. (The symbol $\hat\rtimes$ indicates the
completion of the algebra.) When $\hbar=0$,
the crossed product algebra becomes commuting,
resulting in the algebra $C_0(I) \hat\times \integers$, which
is equal to $C_0(I) \hat\times C_0(S_1)=C_0(I\times S_1)$,
the continuous functions vanishing at infinity on
the cylinder $I\times S_1$.

In \cite{Doku_2005}) the crossed product $C^*$-algebra of
\cite{Sieben_1996} is constructed based on the algebraic
crossed product algebra as also defined in the second
part in definition \ref{def_alg_cpa}. Covariant representations of
$\text{Cyl}(I, \alpha_\hbar)$ are also covariant
representations of $\text{Cyl}^*_{\alpha_\hbar}(I)$.
This makes it possible to define "fuzzy spaces" as limits
of matrix algebras, which at the same time are
representations of the crossed product algebras as defined
herein and as in \cite{Sieben_1996}.

\subsection{Star product}

We define a linear mapping 
\begin{align*}   
  \Psi: \text{Cyl}(I, \alpha_\hbar) \rightarrow S(I\times S_1), \qquad 
  \sum_{n\in \integers}^{\text{finite}} f_n \delta_n \mapsto \sum_{n\in \integers}^{\text{finite}} f_n(x) e^{in\varphi}
\end{align*}  
mapping elements of the $\alpha$-fuzzy cylinder
to the function algebra on the cylinder $S(I\times S_1)$, 
composed of functions, which can be expanded in finite Fourier series
with respect to the coordinate $\varphi$ of the circle $S_1$ and where
the Fourier coefficients solely have support in $I_n$ (see proposition
\ref{alp_fuz_cyl_alg} for the definition of $I_n$).
Note that the mapping $\Psi$ is invertible.
\begin{align*}   
  \Psi^{-1}:   S(I\times S_1) \rightarrow \text{Cyl}(I, \alpha_\hbar), \qquad 
   f  \mapsto \sum_{n\in \integers} \frac{1}{2\pi} \int^{\pi}_{-\pi}  f(x,\varphi)e^{-in\varphi}d\varphi \, \delta_n  
\end{align*} 
Since per definition, there are solely a finite number of Fourier coefficients for the function $f$,
the sum at the right hand side is finite.

With the help of the algebra product of $\text{Cyl}(I, \alpha_\hbar)$,
we can define a $\star$ product on the function algebra $S(I\times S_1)$.

\begin{proposition} \label{star_prod}
  Let $f,g\in S(I\times S_1)$ be two functions on the cylinder $I\times S_1$.
  Then
  \begin{align*}   
      f\star_{\alpha_\hbar} g = \Psi\paraa{Psi^{-1}(f)\Psi^{-1}(g) }
       = \sum_{n, m\in \integers}^{\text{finite}} f_n \paraa{ g_m|I_{-n}\circ \alpha_\hbar^{-n} } e^{i(n+m)\varphi}
  \end{align*}  
  is an associative product on the function algebra $S(I\times S_1)$ 
\end{proposition}
\begin{proof}
  Associativity of the star product follows directly from the associativity of the algebra product.
  The sum form of the product follows from proposition \ref{alp_fuz_cyl_alg}.
\end{proof}

\begin{proposition} \label{pois_str}
  Such as in proposition \ref{star_prod}, 
  let $f,g\in S(I\times S_1)$ be two functions, which are
  additionally differentiable in the coordinate $x\in I$.
  Furthermore, let $\alpha_\hbar$ be differentiable in $\hbar$
  and let $\beta(x)=\frac{\partial \alpha_0(x)}{\partial\hbar}$
  be the first derivative of $\alpha$ at $\hbar=0$.
  Then 
  \begin{align*}   
      f\star_{\alpha_\hbar} g = fg - i\hbar \beta \partial_\varphi f \partial_x g +\mathcal{O}(\hbar^2)
  \end{align*}
  inside its support.
\end{proposition}
\begin{proof}
   Since $\alpha_\hbar$ is differentiable in $\hbar$ it follows that
   $\alpha_\hbar(x)=x+\hbar \beta(x)+\mathcal{O}(\hbar^2)$.
   Thus inside the support of $f\star_{\alpha_\hbar} g$
  \begin{align*}   
      f\star_{\alpha_\hbar} g 
        & = \sum_{n, m\in \integers}^{\text{finite}} f_n \paraa{ g_m +\hbar n \beta \partial_x g_m} e^{i(n+m)\varphi} +\mathcal{O}(\hbar^2) \\
        & = fg - i\hbar \beta \partial_\varphi f \partial_x g +\mathcal{O}(\hbar^2)
  \end{align*}  
  
\end{proof}

As the first Taylor coefficient of the commutator of a star product is 
a Poisson structure, it follows that 
\begin{align*}   
      \{f, g \}_{\alpha_\hbar} = \beta \paraa{ \partial_x f \partial_\varphi g -  \partial_\varphi f \partial_x g }
\end{align*}
is the Poisson structure associated with the star product of
proposition \ref{star_prod}.

\section{Special subalgebras} \label{sec:spec_sub} 

In the following we assume that the function
algebra $\mathcal{\tilde{F}}(I)$ defining the $\alpha$-fuzzy
cylinder $\text{Cyl}(I, \alpha_\hbar)$ is the algebra 
$\mathcal{F}(I)$ of all functions on $I$. Therefore the
partial identity functions $p_1$ exist and we can
use proposition \ref{alp_fuz_cyl_pi} and in particular the
expansion in the algebra elements $U$ and $U^*$.

Based on this we define subalgebras of $\alpha$-fuzzy
cylinder $\text{Cyl}(I, \alpha_\hbar)$, which are generated by 
the algebra element
\begin{align} \label{alg_gen}
   A= f U
\end{align} 
with a function $f\in \mathcal{F}(I)$. Note that we
consider $*$-subalgebras and therefore the subalgebra
is generated by $A$ and $A^*$.

We introduce the real and non-negative function $\phi=f \overline{f}$,
and the commutator relations of $A$, $A^*$ become
\begin{align} 
   [ A, A^*] = p_1 \phi - (p_{1} \phi) \circ\alpha_\hbar, \qquad  \label{AA_comm}
   [A, g] =  A \paraa{ g - (p_{1} g) \circ\alpha_\hbar}
\end{align}
with a function $g\in \mathcal{F}(I)$. Here, we have
used that
$p_{-1}\paraa{(p_{1} \phi) \circ\alpha_\hbar}=(p_{1} \phi) \circ\alpha_\hbar$,
since $(p_{1} \phi) \circ\alpha_\hbar$ already has domain $I_{-1}$.

We also can form the anticommutator of $A$ and $A^*$
resulting in
\begin{align} \label{AA_anti_comm} 
    A A^*+A^* A = p_1 \phi + (p_{1} \phi) \circ\alpha_\hbar
\end{align} 

With the aid of the algebra elements $A$ and $A^*$ (\ref{alg_gen})
we intend to construct algebras, which have the following
algebra relations
\begin{align} \label{AA_rel} 
   [ A, A^*] = \hbar C(R^2), \qquad  A A^*+A^* A = 2 R^2
\end{align} 
where $R^2$ is a non-negative function on $I$ and $C$ is a function in $R^2$.
Here we have in mind that the algebra elements $A$ and $A^*$ 
correspond to the functions $z$ and $\overline{z}$,
which parametrize the cylinder in the form of a ring in the complex plane.
In this case, the second equation of (\ref{AA_rel})
corresponds to $z\overline{z}=r^2$, 

From (\ref{AA_rel}) it follows that
\begin{align} 
   [ A, A^*] = \hbar C(\frac{1}{2}\paraa{A A^*+A^* A})
\end{align} 

From equations (\ref{AA_comm}), (\ref{AA_anti_comm}) and (\ref{AA_rel})
it follows that
\begin{align*}
    (p_{1} \phi) \circ\alpha_\hbar = R^2 - \frac{\hbar}{2}C(R^2), \qquad 
    p_1 \phi = R^2 + \frac{\hbar}{2}C(R^2)
\end{align*} 
By applying the second equation to the first, we 
arrive at an equation for $R$, $C$ and $\alpha_\hbar$.
\begin{align} \label{2gen_eq}
   R^2 \circ\alpha_\hbar+\frac{\hbar}{2}C(R^2) \circ\alpha_\hbar
        = R^2 - \frac{\hbar}{2}C(R^2)
\end{align} 
which, however, is solely defined on the domain of $\alpha_\hbar$,
i.e. $I_{-1}$.

\begin{proposition} \label{2d_alg}
  Let $\alpha_\hbar$ be a partial bijection with classical limit of the interval $I\subset \reals$,
  wherein $x=\alpha_\hbar(u)$ for $u\in I_{-1}$ is a solution
  of the equation
  \begin{align} \label{x_u_eq}
    x+\frac{\hbar}{2}C(x) = u - \frac{\hbar}{2}C(u)
  \end{align} 
  where $C$ is a function on $I$.
  Further let $\rho:J\rightarrow I$ be a bijection, where
  $J\subset\reals$ is an interval, and let
  \begin{align*} 
    \phi_\rho = \rho+\frac{\hbar}{2}C(\rho), \qquad
    \alpha_{\rho,\hbar} = \rho^{-1} \circ  \alpha_\hbar \circ \rho:
        J_{-1}=\rho^{-1}\para{I_{-1}}\rightarrow J_1=\rho\paraa{I_1}
  \end{align*} 
   Then the subalgebra of the $\alpha$-fuzzy cylinder
  generated by $\sqrt{\phi_\rho}U$, has algebra relations
  \begin{align*} 
     [ A, A^*] =
      \begin{cases}
  		\rho+\frac{\hbar}{2}C(\rho), & u\in J_{1} \setminus J_{-1} \\
             \hbar C(\rho),  & u\in J_1 \cap J_{-1}  \\
	       - \rho + \frac{\hbar}{2}C(\rho), & u\in J_{-1} \setminus J_{1}
       \end{cases} , \,
     \frac{1}{2}\paraa{A A^*+A^* A} = 
      \begin{cases}
  		\frac{1}{2}\paraa{\rho+\frac{\hbar}{2}C(\rho) }, & u\in J_{1} \setminus J_{-1} \\
             \rho,  & u\in J_1 \cap J_{-1}  \\
	       \frac{1}{2}\paraa{\rho-\frac{\hbar}{2}C(\rho) } , & u\in J_{-1} \setminus J_{1}
       \end{cases} 
  \end{align*} 
\end{proposition}
\begin{proof}
  From the definition of concatenation of partial
  bijections, it follows that
  \begin{align*} 
    	\alpha_{\rho,\hbar}:J_{-1}=\rho^{-1}\para{I_{-1}}\rightarrow J_1=\rho\paraa{I_1}
  \end{align*} 
  We evaluate
  $p_1\phi_\rho \pm (p_1\phi_\rho)\circ \alpha_{\rho,\hbar}$
  (see equations (\ref{AA_comm}), (\ref{AA_anti_comm})  )
  on $u \in J_1 \setminus J_{-1}$,
  $u\in J_1\cap J_{-1}$ 
  and $u\in J_{-1} \setminus J_1$.

  It is
  \begin{align*} 
     p_1\phi_\rho =& \begin{cases}
  		       0 , & u\in J_{-1} \setminus J_1 \\
	              \rho+\frac{\hbar}{2}C(\rho), & u\in J_1
                 \end{cases} \\
     (p_1\phi_\rho)\circ \alpha_{\rho,\hbar} = &
           \begin{cases}
  		  \paraa {\text{id}_I+\frac{\hbar}{2}C } \circ \alpha_{\hbar} \circ \rho
  		    = \paraa{ \text{id}_I-\frac{\hbar}{2}C } \circ \rho , & u\in J_{-1}  \\
	        0, & u\in J_1 \setminus J_{-1}
            \end{cases}
  \end{align*} 
  where in the second equation above,
  equation (\ref{x_u_eq}) has been used.
  Recall that $p_1$ is defined with
  respect to $\alpha_{\rho,\hbar}$, i.e. it
  is the partial identity on $J_1$.
  
  It follows
    \begin{align*} 
       p_1\phi_\rho \pm (p_1\phi_\rho)\circ \alpha_{\rho,\hbar} =
      \begin{cases}
  		\rho+\frac{\hbar}{2}C(\rho), & u\in J_{1} \setminus J_{-1} \\
             2\rho,  & u\in J_1 \cap J_{-1} \text{ and "+"} \\
             \hbar C(\rho),  & u\in J_1 \cap J_{-1} \text{ and "-"} \\
	       \pm \paraa{ \rho-\frac{\hbar}{2}C(\rho) }, & u\in J_{-1} \setminus J_{1}
       \end{cases} 
  \end{align*} 
  
\end{proof}

We see that we are solely free in $J_1 \cap J_{-1}$ 
in choosing $C$, which then determines the commutator
and anticommutator in the rest of $J$.

\begin{corollary} \label{2d_alg_prop}
  The algebra of proposition \ref{2d_alg} has
  algebra relations
  \begin{align*} 
     AA^* & =0, & u\in J_{1} \setminus J_{-1} \\
     [ A, A^*] &= \hbar C\parab{\frac{1}{2}\paraa{A A^*+A^* A}}, & u\in J_1 \cap J_{-1} \\
     A^* A & = 0, & u\in J_{-1} \setminus J_{1}
  \end{align*} 	
\end{corollary}

\hfill \\ 

For having a defined classical limit, it is necessary
that the functions of the commutator and the anti-commutator
are continuous and are $0$ at the borders of $J$, such that
$\alpha_\hbar$ maps them to functions in the corresponding
ideals. 

\begin{proposition} \label{2d_alg_cont_1}
  In the nomenclature of proposition \ref{2d_alg},
  assume that $J_{1} \setminus J_{-1}$ is not empty.
  Let $u_0$ and $u_1$ be the different border points of
  $J_{1} \setminus J_{-1}$, where $u_0$ is the 
  border of $J$. Then $u_0$ is the image of $u_1$ by
  $\alpha_{\rho,\hbar}$, i.e.
  $u_0=\alpha_{\rho,\hbar}(u_1)$.
  In this case, the functions $[ A, A^*]$ and 
  $\frac{1}{2}\paraa{A A^*+A^* A}$ are continuous
  at $u_1$, iff these functions are $0$ at $u_0$. 
  
  If these two conditions are fulfilled, then $\rho(u_0)+\frac{\hbar}{2}C\paraa{\rho(u_0)}=0$.
  If $\rho$ is a bijection of $I$ with $\rho(u_0)=u_0$,
  i.e. it is orientation preserving, this results in
  $u_0+\frac{\hbar}{2}C(u_0)=0$.
\end{proposition}
\begin{proof}
  Since $J=J_1\cup J_{-1}$, $J_{1} \setminus J_{-1}$
  is at a side of $J$, such that a border of $J_1$
  and optionally a border of $J_{-1}$ are at the border
  of $J$. Let $u_0$ be the point, which is the border
  of $J$, of $J_{1}$ and optionally $J_{-1}$ at this side
  of $J$. Since $J=J_1 \cup J_{-1}$ and $J_{1} \setminus J_{-1}$
  is not empty, it follows that $J_{-1}$ is inside of
  $J_1$ at this side of $J$, and $u_1$, the border point
  of $J_{-1}$ at this side, is different from $u_0$.
    
  Since $\alpha_{\rho,\hbar}$ maps $J_{-1}$ to $J_1$,
  it maps the respective border of $J_{-1}$ to $J_1$
  and therefore $u_0=\alpha_{\rho,\hbar}(u_1)$.

  The commutator being $0$ at $u_0$ and being
  continuous at $u_1$ means
  \begin{align*} 
      \rho(u_0)+\frac{\hbar}{2}C\paraa{\rho(u_0)} = 0, \qquad 
      \rho(u_1)+\frac{\hbar}{2}C\paraa{\rho(u_1)} = \hbar C \paraa{\rho(u_1)}
  \end{align*} 
  The first equation is equivalent to
   \begin{align*} 
      \rho\paraa{\alpha_{\rho,\hbar}(u_1)} + \frac{\hbar}{2}C\paraa{\rho\paraa{\alpha_{\rho,\hbar}(u_1)}} = 
      \alpha_\hbar\circ\rho(u_1) + \frac{\hbar}{2}C\circ\alpha_\hbar\circ\rho(u_1) = 0
   \end{align*} 
   which is equivalent to the second equation, because of
  (\ref{x_u_eq}).
  
  For the anti-commutator, the two conditions imply
  \begin{align*} 
      \rho(u_0)+\frac{\hbar}{2}C\paraa{\rho(u_0) } = 0, \qquad 
      \rho(u_1)+\frac{\hbar}{2}C\paraa{\rho(u_1) } = 2 \rho(u_1) 
  \end{align*}
  which are also equivalent due to $u_0=\alpha_{\rho,\hbar}(u_1)$
  and (\ref{x_u_eq}).
\end{proof}

\begin{proposition} \label{2d_alg_cont_2}
  In the nomenclature of proposition \ref{2d_alg},
  assume that $J_{-1} \setminus J_{1}$ is not empty.
  Then let $u_0$ and $u_1$ be the different border points of
  $J_{-1} \setminus J_{1}$, where $u_0$ is the 
  border of $I$. Then $u_1$ is the image of $u_0$ by
  $\alpha_{\rho,\hbar}$, i.e.
  $u_1=\alpha_{\rho,\hbar}(u_0)$.
  In this case, the functions $[ A, A^*]$ and 
  $\frac{1}{2}\paraa{A A^*+A^* A}$ are continuous
  at $u_1$, iff these functions are $0$ at $u_0$.
  Here, then $\rho(u_0)+\frac{\hbar}{2}C\paraa{\rho(u_0)}=0$.
\end{proposition}
\begin{proof}
  This can be shown analogously to proposition \ref{2d_alg_cont_1}.
\end{proof}
\begin{remark}
  Note that in the case $J_{-1} \setminus J_{1} \neq \emptyset$
  the function $\rho+\frac{\hbar}{2}C\paraa{\rho}=0$ is not necessarily $0$ at $u_0$.
  However this is necessary, such that it is a member of $C_0(I)$
  and that there is a classical limit. We arrive at the problematic
  condition that
  $\rho(u_1)+\frac{\hbar}{2}C\paraa{\rho(u_1)}=\rho(u_0)+\frac{\hbar}{2}C\paraa{\rho(u_0)}=0$.
\end{remark}

\begin{example} Plane algebra \end{example}
\noindent  For a constant commutator $C(u)=\pm 1$,
equation (\ref{x_u_eq}) is
\begin{align*} 
    \alpha_\hbar(u) \pm \frac{\hbar}{2}=u \mp \frac{\hbar}{2}
\end{align*} 
Thus, $\alpha_\hbar(u)=u\mp\hbar$. The algebra
generator becomes $A=\sqrt{\phi_\rho}U=\sqrt{\rho\pm \frac{\hbar}{2}}U$
and we are interested in an interval $I=[a,\infty)$
We have to differentiate between two cases:

\hfill \\
Case 1, $C=+1$: 
$\alpha_\hbar$ maps from $I_{-1}=[a+\hbar,\infty)$
to $I_{1}=I=[a,\infty)$.
We only have to consider
$I_1 \setminus I_{-1}=[a, a+ \hbar[$ and
$I_1 \cap I_{-1}=I_{-1}$, since $I_{-1}\setminus I_1 = \emptyset$.
We arrive at
 \begin{align*} 
     [ A, A^*] =
      \begin{cases}
  		\rho+\frac{\hbar}{2}, & u\in \rho^{-1}\paraa{[a, a+\hbar[} \\
             \hbar,  & u\in \rho^{-1}\paraa{[a+\hbar,\infty)} 
       \end{cases} , \\
     \frac{1}{2}\paraa{A A^*+A^* A} = 
      \begin{cases}
  		\frac{1}{2}\paraa{\rho+\frac{\hbar}{2}}, & u\in \rho^{-1}\paraa{[a, a+\hbar[}\\
             \rho,  & u\in \rho^{-1}\paraa{[a+\hbar,\infty)}   
       \end{cases} 
\end{align*} 
These functions are $0$ at
$u_0=\rho^{-1}\paraa{a}$, when
$a+\frac{\hbar}{2}=0$, i.e. $a=-\frac{\hbar}{2}$.
We can apply proposition \ref{2d_alg_cont_1}
and these functions are also continuous at
$u_1=\rho^{-1}\paraa{a+\hbar}$, 
which we can check directly here.
$\phi_\rho=\rho+\frac{\hbar}{2}$ is always
non-negative and the algebra generator $A$ is well
defined.

\hfill \\
Case 2, $C=-1$: 
$\alpha_\hbar$ maps from $I_{-1}=I=[a,\infty)$
to $I_{1}=[a+\hbar,\infty)$ for an $a\in \reals$.
We only have to consider
$I_{-1} \setminus I_{1}=[a, a+\hbar[$ and
$I_1 \cap I_{-1}=I_{1}$, since $I_{1}\setminus I_{-1} = \emptyset$.
We arrive at
 \begin{align*} 
     [ A, A^*] =
      \begin{cases}
             -\hbar,  & u\in \rho^{-1}\paraa{[a+\hbar,\infty)}  \\
	       - \paraa{\rho + \frac{\hbar}{2}}, & u\in \rho^{-1}\paraa{[a, a+\hbar[}
       \end{cases} , \\
     \frac{1}{2}\paraa{A A^*+A^* A} = 
      \begin{cases}
             \rho,  & u\in \rho^{-1}\paraa{[a+\hbar,\infty)}  \\
	       \frac{1}{2}\paraa{\rho+\frac{\hbar}{2} } , & u\in \rho^{-1}\paraa{[a, a+\hbar[}
       \end{cases} 
\end{align*} 
The commutator and anti-commutator become $0$
at $\rho^{-1}(a)$, when $a=-\frac{\hbar}{2}$,
which according to proposition \ref{2d_alg_cont_2} 
implies continuity at $\rho^{-1}(a+\hbar)$.
However, we have $\phi_\rho=\rho-\frac{\hbar}{2}$, i.e.
$\phi_\rho\paraa{\rho^{-1}(a)}=-\hbar$, independently of $a$.

\begin{example} Poincare disc algebra \end{example}
\noindent For $z\overline{z}<1$, the metric of
the Poincare disc is 
\begin{align*} 
   ds^2 = 4 \frac{dz d\overline{z} }{\paraa{1-z \overline{z}}^2}
\end{align*}
which induces the symplectic form
\begin{align*} 
   d\omega = 2 \frac{dz \wedge d\overline{z} }{\paraa{1-z \overline{z}}^2}
\end{align*}
(see for example \cite{Klim_1992}, in which also an algebra with relations
based on this Poisson structure is defined).

As the inverse of the volume form is a Poisson structure,
a compatible commutator for the Poincare disc is $C(u)=\frac{1}{2}\paraa{1-u}^2$
and \ref{x_u_eq} becomes
\begin{align*} 
    \alpha_\hbar + \frac{\hbar}{4}(1-\alpha_\hbar)^2 = u-\frac{\hbar}{4}(1-u)^2
\end{align*} 
Assuming $\hbar\neq 0$, the solutions of this equation are
\begin{align*} 
  \alpha_\hbar(u)  & = \frac{1}{\hbar}\parab{ \hbar-2 \pm \sqrt{4 - 4\hbar(1-u)-\hbar^2(1-u)^2}  } \\
     & = 1- \frac{2}{\hbar} \pm \frac{2}{\hbar} \sqrt {1 + \frac{\hbar}{2} 2(u-1) - \paraa{\frac{\hbar}{2}}^2(u-1)^2 }
\end{align*} 
We restrict to the "+"-case, which has the correct limit 
$\alpha_\hbar(u)\rightarrow u$ for $\hbar\rightarrow 0$.
The expression inside the root is solely non-negative for
\begin{align*} 
   1- \frac{2\sqrt{2}-2}{\hbar} \leq u \leq 1 + \frac{2\sqrt{2}+2}{\hbar}
\end{align*} 
This means that for $\hbar<2\sqrt{2}-2$ $\alpha_\hbar$ 
is defined at least inside the interval $[0,1]$,
which we are interested in. Furthermore,
$\alpha_\hbar(1)=1$, independently from $u$ and $\hbar$.
We will consider $I=[a,1]$ with $a\leq 0$,
where we assume that $\hbar$ is sufficiently small,
such that $\alpha_\hbar$ is defined inside $I$.

$\alpha_\hbar$ has a maximum at $u=1+\frac{2}{\hbar}>1$. Left
of this value it is strictly increasing. Thus, inside $I$, $\alpha_\hbar$
is strictly increasing. Additionally, $\alpha_\hbar(u)<u$ inside $I$.

In the vicinity of $u=0$ we have
\begin{align*} 
   v                 & = 1 +\frac{2}{\hbar} - \frac{1}{\hbar} \sqrt{4+4\hbar-\hbar^2}
           \thickapprox   \frac{\hbar}{2}, \text{ with } \alpha_\hbar(v)=0, \text{ and} \\
   \alpha_\hbar(0)  & = 1 - \frac{2}{\hbar} + \frac{1}{\hbar} \sqrt{4-4\hbar-\hbar^2}
           \thickapprox   -\frac{\hbar}{2}
\end{align*} 
From the second equation it
follow that  $\alpha_\hbar$ is approximately the 
shift $u\mapsto u-\frac{\hbar}{2}$
in the vicinity of $u=0$ and we are in the situation 
of proposition \ref{2d_alg_cont_1} there.
Since $\alpha_\hbar(1)=1$ on the
other side of the interval $I=[a,1]$, we have
avoided the situation of 
proposition \ref{2d_alg_cont_2}. 

We note that the expression inside the square root is
non-negative, when
\begin{align*} 
 \rho \geq \rho_0 = 1-\frac{2}{\hbar}+\frac{2}{\hbar}\sqrt{1-\hbar} \thickapprox -\frac{\hbar}{4}
\end{align*} 
In view of proposition \ref{2d_alg_cont_1}, we
set $a=\rho_0$.

Proposition \ref{2d_alg} tells us the algebra
relations for the subalgebra of the $\alpha$-fuzzy cylinder,
which is generated by the algebra element
$A=\sqrt{\rho+\frac{\hbar}{4}(1-\rho)^2}U$.
The partial bijection $\alpha_\hbar:I_{-1}\rightarrow I_1=I$ maps
the interval $I_{-1}=[\alpha^{-1}_\hbar(a),1]$ to the interval
$I_1=I=[a,1]$. In the notation of proposition \ref{2d_alg} we have
\begin{align*} 
    J_{1} \setminus J_{-1}= \rho\paraa{[a, \alpha^{-1}_\hbar(a)[}, \qquad
    J_1 \cap J_{-1} =\rho\paraa{I_{-1}}=\rho{[\alpha^{-1}_\hbar(a), 1]}, \qquad
    J_{-1} \setminus J_{1} = \emptyset
\end{align*} 
As already mentioned, we have to differentiate
two cases in the commutator relations
\begin{align*} 
   [ A, A^*]  & =
    \begin{cases}
      \rho+\frac{\hbar}{4}(1-\rho)^2 , & u\in \rho\paraa{[a, \alpha^{-1}_\hbar(a)[} \\
      \frac{\hbar}{2}(1-\rho)^2,  & u\in \rho{[\alpha^{-1}_\hbar(a), 1]}  
    \end{cases} \\
  \frac{1}{2}\paraa{A A^*+A^* A} & = 
   \begin{cases}
       \frac{1}{2}\paraa{\rho+\frac{\hbar}{4}(1-\rho)^2 }, & u\in \rho\paraa{[a, \alpha^{-1}_\hbar(a)[} \\
       \rho,  & u\in \rho{[\alpha^{-1}_\hbar(a), 1]}  
    \end{cases} 
\end{align*} 

Due to proposition \ref{2d_alg_cont_1} and our choice of $a$,
we know that these functions are continuous within $J=\rho^{-1}(I)$.

\newpage
\part{Crossed product algebras for partial group actions}

In the first part, will solely use the results of the last section \ref{sec:CPA_PGA} of this second part,
in particular the crossed product algebra of a partial group action of $\integers$ on
the algebra of functions on an interval.

To understand this construction, we have collected the relevant aspects of
inverse semigroups, their actions, partial group actions and crossed
product algebras, which can be constructed from these actions. 
As these results are scattered over some articles and books, and in particular are
mixed with results on $C^*$-algebras, we have collected them and present them
in the following in a self-contained way, such that they can be understood
without nearly no background knowledge. Most of the results are already 
known. What is new according to our knowledge are the consideration
about partial identities (section \ref{sec:CPA_PI}) and the application
of the crossed product algebras constructed herein on function 
algebras.

Although we solely need partial group actions and their crossed 
products in the first part, we firstly introduce inverse semigroups
and their actions in sections \ref{sec:inv_semi} and
\ref{sec:semi_act}. In sections \ref{sec:CPA_PI},
\ref{sec:CPA_NDI} and \ref{sec:cov_rep} we then
define crossed product algebras and covariant representations
for inverse semigroup actions. Here, the special treatment
of partial identities in section \ref{sec:CPA_PI} is new,
which substantially simplifies the corresponding
crossed product algebras, when such elements are present.

In section \ref{sec:spec_cov_rep} a special covariant
representation is defined, which is the basis for all
representations used in the first part.

In section \ref{sec:pga}, we clarify the relationship between
partial group actions and inverse semigroups and
then introduce in section \ref{sec:CPA_PGA} the
crossed product algebras based on partial group actions
of $\integers$ on function algebras of functions on subintervals
of the real line, such as needed in the first part.

\section{Inverse semigroups} \label{sec:inv_semi}

\cite{Lawson_1998} is a good introduction, in which the
missing proofs in the following can be found. There are
also short introductions by the same author available in the Internet.

\begin{definition}
A semigroup $S$ is an \emph{inverse semigroup}, if 
for each $s\in S$ there exits a unique element $s^*$ such that
\begin{align*} 
  s=ss^*s, \qquad s^*=s^*ss^*
\end{align*}
\end{definition}
The inverse operation is an antiautomorphism of the
inverse semigroup $S$, i.e.
\begin{align*} 
  (s^*)^*=s, \qquad (st)^*=t^*s^*
\end{align*}
for $s,t\in S$.

Note that $ss^*$ and $s^*s$ are \emph{idempotents}, i.e.
elements $q$ such that $q^2=e$. An inverse semigroup can
have a unit (or identity) element and a zero element,
which both are idempotents. 

Idempotents are important elements of a semigroup and
also characterize the difference between groups and inverse
semigroups. All groups are inverse semigroups. An inverse semigroup
is a group, if and only if it has a unique idempotent, which
is the unit.

All idempotents $q$ commute and therefore $q=q^*=qq^*$.
If $q$ is an idempotent then also $sqs^*$ is an idempotent for any
$s\in S$. 

An important semigroup is the \emph{symmetric inverse monoid}
$I(X)$ of a set $X$. (Inverse semigroups with unit element
are also called \emph{inverse monoids}.) The symmetric inverse monoid
is the set of all \emph{partial bijections} from the set $X$ into itself.
A partial bijection $\alpha:A\rightarrow B$ is a bijective map
between two subsets $A$ and $B$ of $X$. $A=\text{dom}(\alpha)$
is the domain of $\alpha$ and $B=\text{ran}(\alpha)$ is its range.
The multiplication in $I(X)$ is the composition of partial bijections in the
possibly largest domain, i.e. if $\alpha$ and $\beta$ are two partial
bijections, then $\alpha\beta$ has domain
$dom(\alpha\beta)=\beta^{-1}(\text{ran}(\beta) \cap \text{dom}(\alpha) )$
and range $\text{ran}(\alpha\beta)=\alpha(\text{ran}(\beta)\cap \text{dom}(\alpha) )$.

\begin{proposition}
  The composition of partial bijections is associative, i.e.
  $\text{dom}\paraa{(\gamma\cdot\beta)\cdot\alpha}= \text{dom}\paraa{\gamma\cdot(\beta\cdot\alpha)}$
  and
  $\text{ran}\paraa{(\gamma\cdot\beta)\cdot\alpha}= \text{ran}\paraa{\gamma\cdot(\beta\cdot\alpha)}$
\end{proposition}

$I(X)$ contains a zero element, which is the
empty bijection. For example, when $ran(\beta)\cap dom(\alpha)=\emptyset$
is empty, then $\alpha\beta=0$. The only idempotents of
$I(X)$ are the identity functions $1_A$ on the subsets $A$ of $X$. In
particular $1_A 1_B=1_{A\cap B}$ for two subsets $A$ and $B$.
It is $\alpha\alpha^{-1}=1_{ran(\alpha)}$ and 
$\alpha^{-1}\alpha=1_{dom(\alpha)}$.

Inverse semigroups always can be viewed as inverse subsemigroups of
a symmetric inverse monoid. (An \emph{inverse subsemigroup}
is a subsemigroup in which the inverse mapping is closed.)
\begin{theorem} [Wagner-Preston representation theorem]
  Every inverse semigroup $S$ can be embedded in the symmetric
  inverse monoid $I(S)$, for which $S$ is considered as a set.
\end{theorem}

\section{Semigroup actions} \label{sec:semi_act}

\begin{definition}
An \emph{action} of an inverse semigroup $S$ on the set $X$ 
is a homomorphism $\alpha: S\rightarrow I(X)$ with $X_e=X$.
\end{definition}
In other words, each $\alpha_s:X_{s*}\rightarrow X_s$
is a partial bijective map. Note that we have denoted the
range of $\alpha_s$ with $X_s$ and the domain with $X_{s*}$.

\begin{proposition} \label{action_prop}
  Let $\alpha$ be an action of the inverse semigroup $S$ on
  the set $X$. Then
  \begin{enumerate}[(i)]
  \item $\alpha_e = \text{id}_X$, the identity map on X.
  \item $\alpha_s\alpha_{s^*}=\text{id}_{X_{s}}$, $\alpha_{s^*}\alpha_s=\text{id}_{X_{s^*}}$
                 for $s\in S$
  \item $\alpha_{s^*}=\alpha_s^{-1}$ for $s \in S$
  \item $X_{st}=\alpha_s\paraa{X_t\cap X_{s^*}} \subset X_s$ for $s,t \in S$
  \end{enumerate}
\end{proposition}
\begin{proof}
  (i) $\text{id}_X=\alpha_e \alpha_e^{-1} =\alpha_{ee} \alpha_e^{-1} = \alpha_e \alpha_e \alpha_e^{-1} = \alpha_e$
  
  (ii) Let $p_s=\alpha_s\alpha_{s^*}$, then
  $p_s\alpha_s=\alpha_s\alpha_{s^*}\alpha_s=\alpha_{ss^ss}=\alpha_s$.
  Therefore $p_s=\text{id}_{X_s}$ since $\alpha_s$ is invertible.
  The other part follows by exchanging $s$ and $s^*$.
  
  (iii) This follows directly from (ii).
  
  (iv) By definition 
  $\alpha_{st}:X_{(st)^*}\rightarrow X_{st}$,
  $\alpha_{s}:X_{s^*}\rightarrow X_s$,
  $\alpha_{t}:X_{t^*}\rightarrow X_t$,
  from the definition of the partial bijection
  $\alpha_{st}=\alpha_s\alpha_t$ it follows
    \begin{align*}
     X_{(st)^*}=\alpha_{t^*}\paraa{X_t\cap X_{s^*}} 
        \xrightarrow{\alpha_t} X_t\cap X_{s^*}
        \xrightarrow{\alpha_s} \alpha_s\paraa{X_t\cap X_{s^*}}=X_{st}
  \end{align*}
\end{proof}

\begin{proposition} \label{action_prop_idem}
  Let $\alpha$ be an action of the inverse semigroup $S$ on
  the set $X$ and let $q\in S$ be an idempotent. Then
    \begin{enumerate}[(i)]
      \item $\alpha_q = \text{id}_{X_{q}}$
      \item $X_{qs} = X_q \cap X_s$ for $s\in S$.
    \end{enumerate}
\end{proposition}
\begin{proof}
  (i)  $\alpha_q = \alpha_{qq^*} = \text{id}_{X_{q}}$, see item (ii) of proposition \ref{action_prop}.
  
  (ii)  $X_{qs} = \alpha_q\paraa{ X_s \cap X_{q^*} }= X_s \cap X_q$
  according to item (i) of this proposition and item (iv) of proposition \ref{action_prop}.
\end{proof}

\begin{definition}
  A \emph{partial automorphims} of a algebra $A$ is 
  an isomorphism $\alpha:A_1\rightarrow A_2$ of
  subalgebras $A_1, A_2 \subset A$. 
\end{definition}

\begin{proposition}
  Let $\alpha:A_1\rightarrow A_2$ and 
  $\beta:B_1\rightarrow B_2$ be two partial automorphism
  of the algebra $A$. Then $\beta\alpha$ considered as
  a partial bijection is a partial automorphism.
  
  The set $\text{PAut}(A)$ of partial automorphism of $A$ is
  an inverse semigroup.
\end{proposition}
\begin{proof}
  Considered as a partial bijection, the map $\beta\alpha$ is composed
  as follows
  \begin{align*}
     \alpha^{-1}\paraa{A_2\cap A_1} \xrightarrow{\alpha} A_2\cap A_1 \xrightarrow{\beta} \beta\paraa{A_2\cap A_1}
  \end{align*}
  It is clear that $\beta\alpha(a_1 a_2)=\beta\alpha(a_1)\beta\alpha(a_2)$
  and $\beta\alpha(a_1 + a_2)=\beta\alpha(a_1)+\beta\alpha(a_2)$
  for $a_1,a_2 \in \alpha^{-1}\paraa{A_2\cap A_1}$.

  Furthermore, the intersection of subalgebras $A_2\cap A_1$  
  is a subalgebra of $A$. As the image of a subalgebra of the
  algebra homomorphims $\alpha^{-1}$ and $\beta$ is also
  a subalgebra, the map $\beta\alpha$ is a partial automorphism.
  
  As the set $\text{PAut}(A)$ is a subsemigroup of the inverse
  monoid $I(A)$ and for every $\alpha$ also $\alpha^{-1}$ is
  in the set, $\text{PAut}(A)$ is an inverse subsemigroup of $I(A)$.
\end{proof}

\begin{definition}
An \emph{action} of an inverse semigroup $S$ on the algebra $A$ 
is a homomorphism $\alpha: S\rightarrow PAut(A)$ with $A_e=A$.
\end{definition}
In other words, each $\alpha_s: A_{s*} \rightarrow A_s$
is a partial automorphism of $A$, where we have denoted the
range of $\alpha_s$ with $A_s$ and the domain with $A_{s*}$,
which both are subalgebras of $A$.

Note that the assertions of proposition \ref{action_prop} are also
true for actions on algebras. However, the maps are then partial
automorphism.

\begin{remark}
  For $C^*$-algebras $A$, a partial automorphism is defined
  as an isomorphism of closed ideals of $A$ (see for example
  \cite{Sieben_1996}, \cite{Doku_2005}). Since a closed
  subideal $J\subset I$ of a closed ideal $I$ of $A$ is also
  an ideal of $A$, the set of partial automorphism of
  a $C^*$-algebras $A$ is an inverse semigroup. However,
  a restriction to ideals cannot be made in general.
  The fact that a closed subideal $J\subset I$, i.e. $J$ is 
  an ideal of $I$, of a closed ideal $I$ of $A$, is also
  an ideal of $A$ can be shown with the aid of an
  \emph{approximate identity}, a special property of $C^*$-algebras.
  Below we will introduce a \emph{partial identity},
  which in a sense replaces the approximate identity
  in the context of general algebras.
\end{remark}

The following proposition shows that a semigroup action
on a set gives rise to a semigroup action on the function
algebra on the set. The special crossed product algebra
used in the first part is based on the function algebra of
complex-valued functions on an interval of $\reals$.

\begin{proposition} \label{fun_alg_act}
  Let $\alpha$ be an action of an inverse semigroup $S$
  on the set $X$. Furthermore, let $\mathcal{F}(X)$ be
  the algebra of complex-valued functions on $X$. Then
  the action $\alpha$ induces an action $\hat\alpha$ on
  the commutative algebra $\mathcal{F}(X)$. In particular
  \begin{align*}
      \hat\alpha_s\paraa{f}(x)=f\paraa{\alpha_{s^*}(x)}, \qquad
      \hat\alpha_s\paraa{f}\paraa{X\setminus X_s}=0
  \end{align*}
  for $s\in S$, $x \in X_s$, $\alpha_{s^*}:X_s \rightarrow X_{s^*}$,
  and $f \in\mathcal{F}(X_{s^*})$, where 
  $\mathcal{F}(X_{s^*})\subset\mathcal{F}(X)$ is the
  subalgebra of functions on $X$, which vanish outside
  $X_{s^*}$.
 \end{proposition}
 \begin{proof}
    We first note that $\hat\alpha_s\paraa{f}$ is in $\mathcal{F}(X_s)$,
    since, per definition, it vanishes outside $X_s$.
    
    The pullback $\hat\alpha(f)=f\circ\alpha_{s^*}$ of the function
    $f \in \mathcal{F}(X_{s^*})$ by the bijection $\alpha_{s^*}$
    is an algebra homomorphism
    \begin{align*}
       \hat\alpha_s: A_{s^*}=\mathcal{F}(X_{s^*})\rightarrow A_s=\mathcal{F}(X_s)
     \end{align*}
    from the subalgebra $A_{s^*}$ to the subalgebra $A_s$, 
    since $(fg)(x)=f(x)g(x)$ per definition for all functions $f,g\in \mathcal{F}(X)$.
    
    $\hat\alpha_{s^*}$ is the inverse of $\hat\alpha_s$.
    In particular, $\hat\alpha_{s^*} \hat\alpha_s(f)=f\circ\alpha_{s^*}\circ\alpha_s=f$
    for all functions $f\in A_{s^*}$ and 
    $\hat\alpha_s\hat\alpha_{s^*}(f)=f\circ\alpha_s\circ\alpha_{s^*}=f$ 
    for all functions $f\in A_s$ due to
    (ii) of proposition \ref{action_prop}.
    
    The mapping $\hat\alpha$ is a semigroup homomorphism, 
    since
     \begin{align*}
      \hat\alpha_s \hat\alpha_t (f) = f\circ \alpha_{t^*}\alpha_{s^*}
         =f\circ\alpha_{(st)^*} = \hat\alpha_{st}(f)
     \end{align*}
      for $s,t\in S$ and $f\in A_t$, and
     \begin{align*}
      \hat\alpha_e (f) = f\circ \alpha_e = f = \text{id}_{\mathcal{F}(X)} (f)
     \end{align*}      
     for $f\in \mathcal{F}(X)$.
 \end{proof}

\begin{remark} \label{com_fun_alg}
  When $X$ has a topology, then the continuous
  functions $C(X)$ on $X$ are a subalgebra of
  $\mathcal{F}(X)$ for which the subspaces $C(X_s)$
  are also ideals. In this case, we demand that all
  partial bijections $\alpha_s$ for $s\in S$ are
  continuous, such that the induced action $\hat\alpha$
  maps these ideals onto each other, i.e. 
  $\hat\alpha_s\paraa{f}\in \mathcal{\tilde{F}}(X_{s})$
  for all $f\in \mathcal{\tilde{F}}(X_{s^*})$.
  
  We further can restrict to the subalgebra $C_0(X)$
  of functions, which vanish at infinity. Also in this,
  case the induced action $\hat\alpha$ stays within
  the ideals $C_0(X_s)$ for $s\in S$.
  
  In the following, we will summarize these function
  algebras $\mathcal{F}(X)$,  $C(X)$ and
  $C_0(X)$ with $\mathcal{\tilde{F}}(X)$.
\end{remark}

\section{Crossed product algebras with partial identities} \label{sec:CPA_PI}

In \cite{Sieben_1996}) a $C^*$-crossed product algebra for
an inverse semigroup action on a $C^*$-algebra is constructed. However,
there special properties of $C^*$-algebras, such as approximate
identities are used. In \cite{Doku_2005} it is shown that for a
semiprime $*$-algebra and an action of an inverse semigroup
onto the algebra, an algebraic crossed product $*$-algebra exists.
This also can be extended to $C^*$-algebras. However, it is not possible
to define algebra elements  $U_s$ of the crossed product algebra, which
form a representation of the inverse semigroup in the crossed product
algebra, without using properties of $C^*$-algebras, namely again 
approximate identities.

We will therefore concentrate herein on algebras 
$A$ and semigroup actions $\alpha$, which
have the property that every subalgebra $A_s$
being the range of $\alpha_s:A_{s^*}\rightarrow A_s$
contains a partial identity.

\begin{definition} \label{pi_def}
Let $I$ be a subalgebra of the $*$-algebra $A$. 
A \emph{partial identity} is an element $p\in I$
with the property that $p a=a p=a$ for every
$a\in I$.
\end{definition}

\begin{proposition} \label{pi_prop}
  Let $A$ be a $*$-algebra, $I$ a subalgebra of $A$ and
  $p$ be a partial identity of $I$. Then
   \begin{enumerate}[(i)]
   \item $p$ is unique. 
   \item $p$ is a projector, i.e. $p^2=p$ and $p^*=p$.
   \item $I=pAp$.
  \end{enumerate}   
\end{proposition}
\begin{proof}
(i) Suppose that $p$ and $p'$ are partial
identities of $I$, then $p=p p'=p'$, since both elements
are in $I$.

(ii) Since $p$ is an element of $I$, it is idempotent, i.e. $p^2=p$.

Furthermore, $p a = a p = a$ for all $a\in I$ by 
definition. Applying the involution to this equation results in
$a^* p^*= p^*a^* = a^*$ for all $a^*\in I$. Since
$A$ is a $*$-subalgebra, we can replace $a^*$ with $a$. Since
$p$ is unique, the assertion follows.

(iii) Let $a\in I$,then $a=pa=pap$. Therefore
$I\subset pAp$. Let $b=pap$, then $pb=p^2ap=pap=b$
and $bp=pap^2=pap=b$, i.e. $pAp\subset I$.
\end{proof}

\begin{proposition}
  Let $A$ be a $*$-algebra, $I$ a subalgebra of $A$ and
  $p$ be the partial identity of $I$. Then $I$ is an ideal, iff
  $pa\in I$ and $ap\in I$ for all $a\in A$.
  If $I$ is an ideal, then $ap=pa$ for all $a\in A$.
\end{proposition}
\begin{proof}
  When $I$ is an ideal, then $pa\in I$ and $ap\in I$ for all $a\in A$
  by definition. On the other hand, let $b\in I$. Then for all
  $a\in A$ it follows that  $ab=apb\in I$ and $ba=bpa\in I$,
  since $ap\in I$ and $pa\in I$ and $I$ is a subalgebra.
  
  Furthermore, let $c=ap$ and $d=pa$. Since $I$ is an ideal,
  $c$ and $d$ are in $I$, it follows that $c=pc=pap=dp=d$.
\end{proof}

\begin{remark}
  The subalgebras $A_s=\mathcal{F}(X_s)$ of \ref{fun_alg_act} 
  comprise the function $p_s=1_{X_s}$, which is
  $1$ on $X_s$ and $0$ otherwise. $p_s$ is
  a partial identity. Furthermore, the subalgebras $A_s$
  are ideals, and every such ideal is generated by $p_s$.
\end{remark}

\begin{proposition}
  Let $A$ be an algebra, and $\alpha:A_1\rightarrow A_2$ 
  be a surjective algebra homomorphism between the subalgebras
  $A_1$ and $A_2$ of $A$. Let $p_1$ be a partial identity of
  $A_1$, then $p_2=\alpha(p_1)$ is a partial identity of $A_2$.
\end{proposition}
\begin{proof}
  Let $a_2 \in A_2$. Since $\alpha$ is surjective, there exists
  $a_1 \in A_1$ with $\alpha(a_1)=a_2$. Now
  $a_2=\alpha(a_1)=\alpha(p_1 a_1)=\alpha(p_1) \alpha(a_1) = p_2 a_2$.
  In the same way one shows that $a_2=a_2 p_2$.
\end{proof}

\begin{proposition}
  Let $A$ be an algebra, $A_1$ and $A_2$ be ideals of $A$, 
  $p_1$ be the partial identity of $A_1$, $p_2$ the partial
  identity of $A_2$ and $p_{12}$ be the partial identity of
  the ideal $A_{12}=A_1\cap A_2$. Then $p_{12}=p_1 p_2 = p_2 p_1$.
\end{proposition}
\begin{proof}
  Since $p_{12}\in A_{12}$, which is contained in $A_1$ and $A_2$,
  $p_{12}$ is also an element of $A_1$ and $A_2$. Therefore,
  $p_{12}=p_1 p_{12} =p_1 p_2 p_{12}=p_1 p_2$, since $p_1 p_2\in A_{12}$ (here
  we need that $A_1$ and $A_2$ are ideals).
  Analogously we can show that $p_{12}=p_2 p_1 p_{12}=p_2 p_1$. Since
  $p_{12}$ is unique it follows that $p_1 p_2 = p_2 p_1$.
\end{proof}

Remember that a subalgebra $A_{st}=\alpha_s\paraa{A_{s^*}\cap A_t}$
of an action is in general not $A_s\cap A_t$. Thus, $p_{st}$ is not $p_sp_t$.
However, we can show the following:
\begin{proposition}
   Let $\alpha$ be an action $\alpha_s:A_{s^*}\rightarrow A_s$ of the inverse semigroup $S$ on the 
   $*$-algebra $A$, where each subalgebra $A_s$ is an ideal and has a partial identity $p_s$.
   Then $p_{st}=\alpha_s\paraa{p_{s^*}p_t}$.
\end{proposition}
\begin{proof}
    $A_{s^*}$ is an ideal, thus $p_{s^*}p_t \in A_{s^*}$ 
    and $\alpha_s\paraa{p_{s^*}p_t}$ is defined.

   Since $A_{st}=\alpha_s\paraa{A_{s^*}\cap A_t}$,
   there exists a $b\in A_{s^*} \cap A_t$ for every $a\in A_{st}$ 
   with $a=\alpha_s(b)$. Thus,
   \begin{align*} 
        \alpha_s\paraa{p_{s^*}p_t}a
            =\alpha_s \paraa{p_{s^*}p_t} \alpha_s \paraa{b}
            =\alpha_s\paraa{p_{s^*} p_t b }
            = \alpha_s \paraa{b} = a
   \end{align*}
   since $b\in A_{s^*}$ and $b\in A_t$.
   Analogously, one shows that $a\alpha_s\paraa{p_{s^*}p_t}=a$.
   Therefore $\alpha_s\paraa{p_{s^*}p_t}=p_{st}$, since the partial identity is unique.
\end{proof}

We now construct an algebra, which will be the
basis for the crossed product algebra.
\begin{proposition} \label{cpa_construct}
 Let $\alpha$ be an action $\alpha_s:A_{s^*}\rightarrow A_s$ of the inverse semigroup $S$ on the 
 $*$-algebra $A$, where each subalgebra $A_s$ is an ideal and has a partial identity $p_s$.Then the vector space
   \begin{align*} 
      L(A, S, \alpha) = \left\{ x:S\rightarrow A : x(s) \in A_s, x \text{ has finite support}  \right\}
   \end{align*}
   with multiplication
    \begin{align*} 
      (x \cdot y)(s)=\sum_{rt=s} \alpha_r \paraa{ \alpha_{r^*}\paraa{x(r) } y(t) } 
   \end{align*}
   and involution
    \begin{align*} 
      x^*(s)= \alpha_s\paraa{ x(s^*)^* }
   \end{align*}
   is an associative $*$-algebra.
\end{proposition}
\begin{proof}
  We denote the function, which maps $s\in S$ to $a_s\in A_s$
  and the other elements to $0$, with $a_s\delta_s$. 
  Formally, $\delta_s(t)=\delta_{s,t}$ with the last being the Kronecker delta,
  however since in general $1 \notin A_s$, $\delta_s$ is a symbol, which cannot
  stand alone. 
  
  The multiplication translates into
  \begin{align*} 
    a_s\delta_s \cdot  a_t\delta_t 
       = \alpha_s \paraa{ \alpha_{s^*}\paraa{a_s}a_t } \delta_{st}
       = \alpha_s \paraa{ \alpha_{s^*}\paraa{a_s} p_{s^*} a_t } \delta_{st}
        = a_s \alpha_s\paraa{ p_{s^*}a_t } \delta_{st} 
  \end{align*}
  since $a_s$ is in $A_s$ and $a_t$ is in $A_t$. $\alpha_{s^*}\paraa{a_s}$ is in $A_{s^*}$ and
  therefore $\alpha_{s^*}\paraa{a_s} = \alpha_{s^*}\paraa{a_s} p_{s^*}$. Since
  $p_{s^*} a_t$ is in $A_{s^*}$ (which is an ideal), we can use
  that $\alpha_s$ is an algebra homomorphism
  on $A_{s^*}$ and that $\alpha_{s^*}$ is its inverse on $A_{s^*}$.
    
  For showing associativity it is enough to show that
  \begin{align*} 
      a_r\delta_r \cdot (a_s\delta_s \cdot  a_t\delta_t) = (a_r\delta_r \cdot a_s\delta_s) \cdot  a_t\delta_t
  \end{align*}
  since this extends to finite sums by linearity. The left hand side is
  \begin{align*} 
      a_r\delta_r \cdot (a_s\delta_s \cdot  a_t\delta_t) 
           & = a_r\delta_r \cdot \parab{ a_s \alpha_s\paraa{ p_{s^*}a_t } \delta_{st} } \\
           & = a_r \alpha_r \parab{ p_{r^*} a_s \alpha_s\paraa{ p_{s^*}a_t } } \delta_{rst}    
  \end{align*}
  The right hand side is
  \begin{align*} 
     (a_r\delta_r \cdot a_s\delta_s) \cdot  a_t\delta_t
         & = \parab{ a_r \alpha_r\paraa{ p_{r^*}a_s } \delta_{rs} } \cdot  a_t\delta_t \\
         & = a_r \alpha_r\paraa{ p_{r^*}a_s } \alpha_{rs} \parab{ p_{(rs)^*} a_t} \delta_{rst} \\
         & = a_r \alpha_r\parab{ p_{r^*}a_s \alpha_{s} \paraa{ p_{s^*r^*} p_{s^*} a_t}  } \delta_{rst} \\
         & = a_r \alpha_r\parab{ p_{r^*}a_s \alpha_{s} \paraa{ \alpha_{s^*}\paraa{p_s p_{r^*}} p_{s^*} a_t}  } \delta_{rst} \\
         & = a_r \alpha_r\parab{ p_{r^*}a_s  p_s p_{r^*} \alpha_{s} \paraa{ p_{s^*} a_t}  } \delta_{rst} \\
         & = a_r \alpha_r\parab{ p_{r^*}a_s  \alpha_{s} \paraa{ p_{s^*} a_t}  } \delta_{rst} 
  \end{align*}
  
  It remains to show that $*$ is an anti-algebra homomorphims
  $(x\cdot y)^*=y^*\cdot x ^*$. It is again enough to show
  this for $x=a_s\delta_s$ and $y=a_t \delta_t$. The 
  involution translates into
  \begin{align*} 
      (a_s\delta_s)^* = \alpha_{s^*}(a_s^*)\delta_{s^*}
  \end{align*}
  On the one hand
  \begin{align*} 
     \paraa{a_s\delta_s\cdot a_t \delta_t}^*
        & = \parab{a_s \alpha_s\paraa{ p_{s^*}a_t } \delta_{st} }^*
           =  \alpha_{(st)^*} \parab{ \paraa{ a_s \alpha_s\paraa{ p_{s^*}a_t} }^* }\delta_{(st)^*} \\
        & =  \alpha_{t^*} \parab{ \alpha_{s^*} \paraa{ \alpha_s\paraa{ a_t^* p_{s^*} } a_s^* } } \delta_{(st)^*} 
           =   \alpha_{t^*} \parab{ a_t^* p_{s^*} \alpha_{s^*} \paraa{a_s^* } } \delta_{(st)^*} \\
        & =   \alpha_{t^*} \parab{ a_t^* \alpha_{s^*} \paraa{a_s^* } } \delta_{(st)^*} 
  \end{align*}
 and on the other hand
  \begin{align*} 
     \paraa{a_t\delta_t}^*\cdot \paraa{a_s \delta_s}^*
        & = \alpha_{t^*}(a_t^*)\delta_{t^*} \cdot \alpha_{s^*}(a_s^*)\delta_{s^*}
           = \alpha_{t^*}(a_t^*) \alpha_{t^*} \parab{ p_t \alpha_{s^*}(a_s^*) } \delta_{t^*s^*} \\
        & = \alpha_{t^*} \parab{ a_t^*  p_t \alpha_{s^*}(a_s^*) } \delta_{(st)^*}
           = \alpha_{t^*} \parab{ a_t^* \alpha_{s^*}(a_s^*) } \delta_{(st)^*}
  \end{align*}
\end{proof}
Compare also the proofs of proposition 5.1 of \cite{Sieben_1996} and Theorem 3.1 of \cite{Doku_2005}.

\begin{proposition} \label{cpa_rel}
  Let $A$, $S$, $\alpha$ and $L(A, S, \alpha)$ be as
  in proposition \ref{cpa_construct}, where
  $S$ has a unit element $e$. Then
     \begin{enumerate}[(i)]
      \item Every element $x \in L(A, S, \alpha)$ is 
                     of the form $x=\sum_{s\in S}^{\text{finite}} a_s U_s$,
                     where $a_s \in A$ and $U_s=p_s\delta_s$ is the function mapping
                     $s\in S$ to $p_s$, the partial identity of $A_s$, and is
                     zero on the other elements of $S$.
      \item $\iota: A\rightarrow L(A, S, \alpha), a\mapsto a\delta_e = a U_e$
      			  is an injective algebra morphism. Therefore, we identify 
      			  $a\in A$ with $\iota(a)$.
      \item $U_s a = \alpha_s\paraa{p_s a} U_s$ for $s\in S$ and $a\in A$.
      \item $(U_s)^* = U_{s^*}$ for $s\in S$
      \item $U_s U_t = U_{st}$
    \end{enumerate}
\end{proposition}
\begin{proof}
 (i) Per definition elements of $L(A, S, \alpha)$ are functions from
 $S$ into $A$ with finite support, i.e. of the form $x=\sum_{s\in S}^{\text{finite}} a_s \delta_s$
 with the symbol $\delta_s$ as defined in proposition \ref{cpa_construct}. Thus,
   \begin{align*} 
       x=\sum_{s\in S}^{\text{finite}} a_s \delta_s 
           = \sum_{s\in S}^{\text{finite}} a_s p_s \delta_s 
           = \sum_{s\in S}^{\text{finite}} a_s U_s
  \end{align*}
  Since every $a_s\in A$ and $ap_s\in A_s$ for every $a\in A$,
  we also can drop the requirement that $a_s\in A_s$ and
  can generalize to $a_s \in A$.
  
 (ii) Note that $p_e$ is the partial identity on $A_e=A$ and $\alpha_e$ is the identity map
    \begin{align*} 
      \iota(a)\iota(b) = ap_e\delta_e\cdot b p_e\delta_e = a \alpha_e(p_e b) \delta_e = \iota(ab)
  \end{align*}
  for $a,b\in A$. If $\iota(a)=\iota(b)$ it follows that $a\delta_e=b\delta_e$. 
  $a\delta_e$ is the function, which maps $e$ to $a$ and which is $0$ otherwise.
  Thus, $a=b$ and $\iota$ is injective.
  
 (iii) We apply the definition of the algebra product
     \begin{align*} 
        U_s a = p_s\delta_s \cdot a\delta_e = \alpha_s\paraa{p_s a} U_s
  \end{align*}
  
  (iv) 
    \begin{align*} 
         (U_s)^* = (p_s\delta_s)^* 
            = \alpha_{s^*}\paraa{p_s^*} \delta_{s^*}
            = p_{s^*} \delta_{s^*} = U_{s^*}
   \end{align*}
  
  (v) 
     \begin{align*} 
         U_s U_t =  p_s\delta_s \cdot p_t\delta_t 
           = p_s \alpha_s\paraa{p_{s^*} p_t} \delta_{st}
           = \alpha_s\paraa{p_{s^*} p_t} \delta_{st}
           = p_{st} \delta_{st} = U_{st}
   \end{align*}
\end{proof}

In \cite{Sieben_1996} and \cite{Exel_2010}, for defining the 
crossed product $C^*$-algebra the algebra constructed
such as in proposition \ref{cpa_construct} has to be divided
by an ideal $N$ generated by the relation
$a\delta_r-a\delta_t$ for $a\in A_r$ and $r\leq t$, i.e. $r=qt$
for some idempotent $q$ (see also definition \ref{def_alg_cpa}).
This is necessary such that each representation of
the crossed product $C^*$-algebra becomes a covariant
representation.  We will also show this for the \emph{algebraic}
crossed product algebra defined in definition \ref{def_pi_cpa}.

\begin{remark}
In this context, we remark that $\alpha_{ss^*}=\alpha_s\alpha_{s^*}$
is the identity on $A_s$.  Therefore for $a\in A_s$, it follows that
$\alpha_{ss^*}a=a=p_s a$.
Since a covariant representation maps both the action of the 
semigroup elements and the algebra elements to
vector space automorphims, the crossed product algebra
should take account of this. 
\end{remark}

Before defining the crossed product algebra, we firstly
show that the relation $a\delta_r-a\delta_t$
defining the ideal $N$ is equivalent to two other relations.
\begin{proposition} \label{cpa_quot_rel}
    Let $A$, $S$ and $\alpha$ be as in proposition
    \ref{cpa_construct}, where $S$ has a unit element $e$,
    and let $U_s$ as in proposition \ref{cpa_rel}
    Then the following relations are equivalent:
     \begin{enumerate}[(i)]
        \item $p_q=U_q$ or equivalently $p_q\delta_e= p_q \delta_q$ for all idempotents $q\in S$   
        \item $a \delta_r=a \delta_t$ for all $a\in A_r$ and $r\leq t$, i.e. $r=qt$ with an idempotent $q$
        \item $p_s=U_{ss^*}$ or equivalently $p_s\delta_e = p_{ss^*}\delta_{ss^*}$ for all $s\in S$ 
     \end{enumerate}
\end{proposition}
\begin{proof}
   $(i)\Rightarrow (ii)$: 
   The relation $p_q= U_q$ 
    is equivalent to $p_q\delta_e=p_q\delta_q$ by definition. 
   When $r=qt$, then $A_r=A_{qt}=A_q\cap A_t$, see proposition \ref{action_prop_idem}.
   Thus, when $a\in A_r$, then also $a\in A_q$ and $a\in A_t$.
   It follows
   \begin{align*} 
     a \delta_r= a U_r = a U_{qt} = a U_q U_t
        = a p_q U_t = a p_q p_ t \delta_t = a \delta_t
   \end{align*} 
   
  $(ii)\Rightarrow (iii)$:
  Since $A_{ss^*}=A_s$, it follows that $p_{ss^*}=p_s \in A_{ss^*}$.
  Therefore it is possible to set $t=e$, $r=q=ss^*$ and $a=p_s$.
  (ii) is then $p_s\delta_{ss^*}=p_s\delta_e$.
    
  $(iii)\Rightarrow (i)$:
  The idempotents are a subset of $S$ and have the property
  that $q=qq^*$.
\end{proof}

\begin{definition} \label{def_pi_cpa}
  Let $\alpha$ be an action $\alpha_s:A_{s^*}\rightarrow A_s$ of the inverse semigroup $S$ on the 
  $*$-algebra $A$, where each subalgebra $A_s$ is an ideal and has a partial identity $p_s$. 
  Furthermore, let $N$ be the ideal of $L(A, S, \alpha)$ generated by
  $p_s\delta_e - p_{ss^*}\delta_{ss^*}$ for $s\in S$. Then
  the quotient algebra
 \begin{align*} 
    A \rtimes_\alpha S = L(A, S, \alpha) / N
  \end{align*}  
  is the \emph{algebraic crossed product} of $A$ and $S$
  by the action $\alpha$.
\end{definition}

Proposition \ref{cpa_quot_rel} shows that additionally to the algebra relations of proposition
\ref{cpa_rel}, the crossed product algebra has
the further algebra relation $p_s=U_{ss^*}$ for each
$s\in S$ or $p_q = U_q$ for each idempotent $q\in S$.

\section{Crossed product algebras with non-degenerate or idempotent ideals} \label{sec:CPA_NDI}

In this section a generalization of
the above defined crossed product algebra is presented. The
generalization is formed from an inverse semigroup action
on a semiprime algebra. This crossed product algebra was introduced in
\cite{Doku_2005} and further developed in \cite{Exel_2010},
where also the name "algebraic crossed product" was introduced. 

The following proposition will be the basis for the proof that the crossed product
algebra is associative.
\begin{proposition} \label{ND_IP_property}
  Let $\beta:I\rightarrow J$ be an automorphism of ideals $I,J\subset A$
  of the algebra $A$ and let $R_b:A\rightarrow I$ be defined
  by $R_{b}(c)=\beta^{-1}\paraa{\beta(c)b}$ for $b\in A$
  and $c\in I$.
  If $I$ is non-degenerate or idempotent, then for $a \in A$
  \begin{align*} 
      R_b(ac) =a R_b(c)  
  \end{align*}
\end{proposition}
\begin{proof}
  At first note that $\beta(c)b \in J$ and $R_b$ is well defined, since
  $\beta^{-1}$ can be applied to $\beta(c)b$. 
  
  We assume that $I$ is non-degenerate. Then, for $d\in I$
  \begin{align*} 
      R_b(ac)d = \beta^{-1}\paraa{\beta(ac)b}d =\beta^{-1} \paraa{ \beta\paraa{ \beta^{-1}\paraa{\beta(ac)b}d } }
      = ac \beta^{-1} \paraa{ b\beta(d) } \\
      = a \beta^{-1} \paraa{ \beta\paraa{ c \beta^{-1} \paraa{ b\beta(d) } } }
      = a \beta^{-1}\paraa{\beta(c)b}d = a R_b(c) d
  \end{align*}
  Analogously, one can show that $d R_b(ac) = d a R_b(c)$, in summary, the 
  element $e=R_b(ac) - a R_b(c)\in I$ has the property that $ed=de=0$ for
  all $d\in I$, and since $I$ is non-degenerate it follows that $e=0$.
  
  Now we assume that $I$ is idempotent, i.e. that every $c\in I$
  is a sum of products of elements of $I$. We simply
  write $c=c_1 c_2$ and extend the following result by linearity. Then,
    \begin{align*} 
      R_b(ac) = R_b(a c_1 c_2) =  \beta^{-1}\paraa{\beta(a c_1 c_2)b}
       =  a c_1 \beta^{-1}\paraa{\beta(c_2)b} \\
       = a \beta^{-1} \paraa{ \beta\paraa{c_1 \beta^{-1}\paraa{\beta(c_2)b} } }
       = a \beta^{-1}\paraa{\beta(c_ 1 c_2)b}    
       = a R_b(c)  
  \end{align*}
\end{proof}
Note that $R_a$ is a right-multiplier of $I$, see proposition 2.5
of \cite{Doku_2005}. 

Recall that an algebra $A$ is semiprime, when every non-zero ideal
is non-degenerate. The following is also a characterization of a semiprime
algebra.
\begin{proposition} \label{semiprime}
 Let $A$ be an algebra. $A$ is semiprime, if and only if every
 non-zero ideal of $A$ is either idempotent or non-degenerate.
\end{proposition}
For the proof see \cite{Doku_2005}, proposition 2.6.

\begin{remark}
  When we additionally assume that an ideal $I$
  has a partial identity $p$, then
  every element $a\in I$ can be written as $a=ap=pa$. Thus,
  $I$ is idempotent. It will follow that the next proposition
  is a generalization of proposition \ref{cpa_construct}.
\end{remark}

\begin{proposition}  \label{cpa_idem_construct}
   Let $\alpha$ be an action $\alpha_s:A_{s^*}\rightarrow A_s$ of the inverse semigroup $S$ on the 
   $*$-algebra $A$, where each subalgebra $A_s$ is a non-degenerate or idempotent ideal.
   Then the vector space
   \begin{align*} 
      L(A, S, \alpha) = \left\{ x:S\rightarrow A : x(s) \in A_s, x \text{ has finite support}  \right\}
   \end{align*}
   with multiplication
    \begin{align*} 
      (x \cdot y)(s)=\sum_{rt=s} \alpha_r \paraa{ \alpha_{r^*}\paraa{x(r) } y(t) } = R_{y(t)}^{\alpha_{r^*}}\paraa{x(r)}
   \end{align*}
   and involution
    \begin{align*} 
      x^*(s)= \alpha_s\paraa{ x(s^*)^* }
   \end{align*}
   is an associative $*$-algebra.
\end{proposition}
\begin{proof}
  We denote the function, which maps $s\in S$ to $a_s\in A_s$, with
  $a_s\delta_s$. Note that $\delta_s$ is a symbol, which cannot
  stand alone. The multiplication then translates into
  \begin{align*} 
    a_s\delta_s \cdot  a_t\delta_t = \alpha_s \paraa{ \alpha_{s^*}\paraa{a_s}a_t } \delta_{st} = R_{a_t}^{\alpha_{r^*}}\paraa{a_s} \delta_{st} 
  \end{align*}
  It is $a_s \in A_s$ and $a_t \in A_t$, therefore $\alpha_{s^*}\paraa{a_s} \in A_{s^*}$ and
  $\alpha_{s^*}\paraa{a_s}a_t \in A_{s^*} \cap A_t$, 
  since both $A_{s^*}$ and $A_t$ are ideals.
    
   For showing associativity it is enough to show that
  \begin{align*} 
      a_r\delta_r \cdot (a_s\delta_s \cdot  a_t\delta_t) = (a_r\delta_r \cdot a_s\delta_s) \cdot  a_t\delta_t
  \end{align*}
  since this extends to finite sums by linearity. The left hand side is
  \begin{align*} 
      \alpha_r \Bigl( \alpha_{r^*}(a_r) \alpha_s \paraa{ \alpha_{s^*}(a_s)a_t }  \Bigr) \delta_{rst}
      =       \alpha_r \Bigl( \underline {\alpha_{r^*}(a_r)  R_{a_t}^{\alpha_{r^*}}\paraa{a_s} } \Bigr) \delta_{rst}
  \end{align*}
  The right hand side is
  \begin{align*} 
      \alpha_{rs} \Bigl(  \alpha_{s^*r^*}  \paraa{ \alpha_r \paraa { \alpha_{r^*}(a_r)a_s } } a_t \Bigr) \delta_{rst}
      & = \alpha_{r} \Bigl(  \alpha_s \paraa{ \alpha_{s^*} \paraa { \alpha_{r^*}(a_r) a_s } a_t } \Bigr) \delta_{rst} \\
      & = \alpha_{r} \Bigl( \underline { R_{a_t}^{\alpha_{r^*}} \paraa { \alpha_{r^*}(a_r) a_s } } \Bigr) \delta_{rst}
  \end{align*}
  Since the ideal $A_{s^*}$ is non-degenerate or idempotent, it follows from
  propostion \ref{ND_IP_property} that the underlined expressions above are 
  equal.
  
  It remains to show that $*$ is an anti-algebra homomorphims
  $(x\cdot y)^*=y^*\cdot x ^*$. It is again enough to show
  this for $x=a_s\delta_s$ and $y=a_t \delta_t$. For the symbols, the 
  involution translates into
  \begin{align*} 
      (a_s\delta_s)^* = \alpha_{s^*}(a_s^*)\delta_{s^*}
  \end{align*}
  It follows that 
  \begin{align*} 
    \paraa{a_s\delta_s \cdot a_t\delta_t}^* 
     & = \paraa{ \alpha_s \paraa{ \alpha_{s^*}\paraa{a_s}a_t } \delta_{st} }^* \\
     & = \alpha_{t^*s^*} \Bigl(  \paraa{ \alpha_s \paraa{ \alpha_{s^*}\paraa{a_s}a_t } }^* \Bigr) \delta_{t^*s^*} \\
     & = \alpha_{t^*s^*} \Bigl(  \alpha_s  \paraa{ a_t^* \alpha_{s^*}\paraa{a_s^*} } \Bigr) \delta_{t^*s^*} \\
     & = \alpha_{t^*} \Bigl(  a_t^* \alpha_{s^*}\paraa{a_s^*} \Bigr) \delta_{t^*s^*} \\
     & = \alpha_{t^*} \Bigl(  \alpha_t \paraa { \alpha_{t^*}(a_t^*) } \alpha_{s^*}\paraa{a_s^*} \Bigr) \delta_{t^*s^*} \\
     & = \alpha_{t^*}(a_t^*)\delta_{t^*} \cdot \alpha_{s^*}(a_s^*) \delta_{s^*} \\
     & = (a_t\delta_t)^* \cdot (a_s\delta_s)^*
  \end{align*}
\end{proof}
Compare also the proofs of proposition 5.1 of \cite{Sieben_1996} and Theorem 3.1 of \cite{Doku_2005}.

\begin{definition} \label{def_alg_cpa}
  Let $\alpha$ be an action $\alpha_s:A_{s^*}\rightarrow A_s$ of the inverse semigroup $S$ on the 
  $*$-algebra $A$, where each subalgebra $A_s$ is a non-degenerate or idempotent ideal.
  Furthermore, let $N$ be the ideal of $L(A, S, \alpha)$ generated by 
  $a\delta_s-a\delta_t$ for $a\in A_s$ and $s\leq t$. Then the 
 \emph{algebraic crossed product} of $A$ and $S$ by the action $\alpha$ is 
   \begin{align*} 
      A \rtimes_{\alpha} S = L(A, S, \alpha) / N
  \end{align*}
\end{definition}

Following proposition shows that this definition is a generalization
of definition \ref{def_pi_cpa}.
\begin{proposition}
  In the nomenclature of definition \ref{def_alg_cpa} assume
  that each subalgebra $A_s$ has a partial identity. Then
  the crossed product algebras of definition \ref{def_pi_cpa} and 
  definition \ref{def_alg_cpa} are the same.
\end{proposition}
\begin{proof}
  The definition of the set, the multiplication and the involution of
  the algebra in propositions \ref{cpa_construct} and \ref{cpa_idem_construct}
  is the same. With proposition \ref{cpa_quot_rel} we already
  have shown that the ideals $N$ are equal.
\end{proof}

\section{Covariant representations} \label{sec:cov_rep}

\begin{definition}
 Let $H$ be a Hilbert space. A bounded linear operator $V\in B(H)$ is
 a \emph{partial isometry}, if $V|_{\text{ker}(V)^\perp}$ is an isometry, i.e. for
 every $v\in \text{ker}(V)^\perp$, $||V v||=||v||$. The subspace
 $\text{ker}(V)^\perp$ is the \emph{domain} of the partial isometry $V$,
 while the image of $V$ of its domain is called \emph{range}
 of the partial isometry $V$.
\end{definition} 

The following proposition shows that partial isometries are well suited for
representing elements of inverse semigroups.
\begin{proposition} \label{part_iso_prop}
  Let $V\in B(H)$. The following is equivalent:
  \begin{enumerate}[(i)]
     \item $V$ is a partial isometry.
     \item $V^*$ is a partial isometry.
     \item $VV^*$ is a projection.
     \item $V^*V$ is a projection.
     \item $V^*VV^*=V^*$ 
     \item $VV^*V=V$ 
 \end{enumerate}
 Moreover, $V V^*$ projects onto the range of $V$ and $V^* V$ projects
 onto the domain of $V$.
\end{proposition}
For the proof see \cite{Skou_1998}, for example.

\begin{definition}
A representation of an inverse semigroup $S$
on a Hilbert space $H$ is a map $V:S\rightarrow B(H)$ into the bounded
linear operators on $H$ such that 
\begin{align*} 
  V_{\alpha\beta}=V_\alpha V_\beta, \qquad V_{\alpha^*}=V_\alpha^*
\end{align*}
 If $S$ has a unit $e$, then we require that $V_e=1_H$, where $1_H$ is the
 unit of $H$.
\end{definition} 

Since $V_\alpha V_\alpha^* V_\alpha=V_\alpha$ and
$V_\alpha^* V_\alpha V_\alpha^*=V_\alpha^*$, the elements
$\alpha$ of the inverse semigroup $S$ are mapped to partial
isometries of $H$.

\begin{definition}  \label{def_cov_rep}
 Let $\alpha$ be an action of the inverse semigroup $S$ on the
 algebra $A$. A \emph{covariant representation} is a triple
 $(\pi, V, H)$, where $\pi:A\rightarrow B(H)$ is a 
 representation of $A$ on the Hilbert space $H$ and $s\mapsto V_s$ is a semigroup
 homomorphism from $S$ into an inverse semigroup of partial
 isometries on $H$ such that
   \begin{enumerate}[(i)]
     \item $V_s\pi(a)V_{s^*} = \pi\paraa{\alpha_s(a)}$ for all $a\in A_{s^*}$,
     \item $V_s$ has domain $\pi(A_{s^*})H$ and range $\pi(A_{s})H$
 \end{enumerate}
\end{definition} 

Condition (i) is the same relation as item (iii) of proposition \ref{cpa_rel}. We 
will see that a representation of the algebraic crossed product algebra of
\ref{def_pi_cpa} is equivalent to a covariant representation.
Condition(ii) is important, since it ensures that
the partial isometries $V_s$ map between the subspaces $\pi(A_{s})H$,
which are the images of the subalgebra $A_{s}$ of the algebra 
representation $\pi$.

\begin{proposition} \label{prop_cov_rep}
  Let $(\pi, V, H)$ be a covariant representation of the action $\alpha$ of the
  inverse semigroup $S$ on the algebra $A$. Then for $s,t \in S$, $a_s\in A_s$
  and $a_t\in A_t$
  \begin{enumerate}[(i)]
     \item $V_e=1_H$ 
     \item $V_s V_{s^*}$ projects onto $\pi(A_{s})H$ and 
             $V_{s^*} V_s$ projects onto $\pi(A_{s^*})H$
     \item $V_sV_{s^*} \pi(a_s) = \pi(a_s) V_sV_{s^*} = \pi(a_s)$ 
     \item $V_{s^*}=V_s^*$
     \item $\pi(a_s)V_s\pi(a_t)V_t=\pi\paraa{\alpha_s\paraa{\alpha_{s^*}(a_s)a_t} } V_{st} $ 
     \item $\paraa{\pi(a_s)V_s}^*=\pi \paraa{ \alpha_{s^*}(a_s^*)}V_{s^*}$
 \end{enumerate}
\end{proposition} 
\begin{proof}
  (i) Since $\pi$ is non-degenerate $H=\pi(A)H=\pi(A_e)H$ and $V_e$ is unitary.
     Therefore $1_H=V_e V_e^{-1} =V_{ee} V_e^{-1}= V_e V_e V_e^{-1}=V_e$.
        
  (ii) Since $V$ is an semigroup homomorphism 
       $V_s V_{s^*} V_s V_{s^*} = V_{ss^*ss^*}=V_{ss^*}=V_sV_{s^*}$. The
       first part of the assertions follows from (2) of
       definition \ref{def_cov_rep}. The second part is analogously.
     
  (iii) $V_sV_{s^*}= V_sV_s^*$ is the projection onto $\pi(A_{s})H$     

  (iv) We have to show that $<V_s h,k>=<h,V_{s^*} k>$ for $h,k\in H$.
  We split $h=h_1+h_2, k=k_1+k_2$, where $h_1\in\pi(A_{s^*})H$, 
  $h_2\in\pi(A_{s^*})H^\perp$, $k_1\in\pi(A_s)H$, $k_2\in\pi(A_s)H^\perp$.
  Thus, it is equivalent to show that
  \begin{align*}
    <V_s h_1,k_1+k_2>=<h_1+h_2,V_{s^*} k_1>
  \end{align*}
  Since $<V_sh_1,k_2>=0$ and $<h_2,V_{s^*}k_1>=0$ it is 
  equivalent to show that
  \begin{align*}
    <V_s h_1,k_1>=<h_1,V_{s^*} k_1>
  \end{align*}
  Since $h_1=V_{s^*}l$ for some $l\in\pi(A_s)H$ we have to show that
  \begin{align*}
    <V_s V_{s^*}l,k_1>=<V_{s^*}lh_1,V_{s^*} k_1>
  \end{align*}
  Now due to (2) $V_s V_{s^*}l=l$, $V_{s^*}$ is a partial isometry and both
  sides are equal to $<l, k_1>$.
   
  (v) 
  \begin{align*}
     \pi(a_s)V_s\pi(a_t)V_t & = V_s \pi\paraa{\alpha_{s^*}(a_s) } V_{s^*} V_s \pi(a_t) V_t \\ 
     &= V_s \pi\paraa{\alpha_{s^*}(a_s) a_t} V_t \\
     &= V_s \pi\paraa{\alpha_{s^*}\alpha_{s} \paraa{\alpha_{s^*}(a_s) a_t} } V_t \\
     &= V_s V_{s^*} \pi\paraa{\alpha_{s} \paraa{ \alpha_{s^*}(a_s) a_t} } V_s V_t \\
     & = \pi\paraa{\alpha_s\paraa{\alpha_{s^*}(a_s)a_t } } V_{st}
   \end{align*}
  
  (vi) $\paraa{\pi(a_s)V_s}^*=V_{s^*}\pi(a_s^*) 
            = V_{s^*} V_s \pi \paraa{ \alpha_{s^*}(a_s^*)}V_{s^*} = \pi \paraa{ \alpha_{s^*}(a_s^*)}V_{s^*}$
  
\end{proof}
The above proof is partially based on propositions 3.2 and 5.3 of \cite{Sieben_1996}.

\begin{proposition} \label{prop_cov_rep_pi}
  Let $(\pi, V, H)$ be a covariant representation of the action $\alpha$ of the
  inverse semigroup $S$ on the algebra $A$, for which the ideals $A_s$
  have partial identities. Then $\pi(p_s)=V_{ss^*}=V_sV_{s^*}$ for
  $s \in S$.
\end{proposition} 
\begin{proof}
  Since $p_s=p_{ss^*}$, we can restrict to the idempotent $q=ss^*$.
  Since $p_q$ is a projector (proposition \ref{pi_prop}, (ii)), the same
  applies to $\pi(p_q)$. Let $a_q \in A_q$. Then $a_q=p_q a_q=a_q p_q$ and
  therefore $\pi(a_q)=\pi(p_q) \pi(a_q)=\pi(a_q) \pi(p_q)$. It follows that
  $\pi(p_q)H = \pi(A_q) H$.

  $V_q$ has range and domain $\pi(A_q)H$ and projects onto $\pi(A_{q})H$
  according to the previous proposition \ref{prop_cov_rep}, i.e. $V_q$ 
  and $\pi(p_q)$ have the same properties as projectors and are the same.
\end{proof}

\begin{proposition} \label{cov_rep_is_alg_rep}
  Let Let $(\pi, V, H)$ be a covariant representation
  of the action $\alpha$ of the inverse semigroup $S$
  on the algebra $A$. Then 
  \begin{align*}
     \paraa{\pi\times V}(x)=\sum_{s\in S}^\text{finite} \pi\paraa{x(s)}V_s
  \end{align*}
  is a representation of the algebra $L(A, S, \alpha)$
  of proposition \ref{cpa_idem_construct} (and also of proposition \ref{cpa_construct}).
\end{proposition}
\begin{proof}
  Let $a_s\delta_s$ be the function, which
  maps $s\in S$ to $a_s\in A_s$ and the
  other elements to $0$. Then
  $\paraa{\pi\times V}(a_s\delta_s)=\pi(a_s)V_s$.
  It is enough to show the properties of an
  algebra homomorphism on the functions $a_s\delta_s$,
  since they extend to general functions by linearity.
  
  For the functions $a_s\delta_s$ and $a_t\delta_t$
  it follows
  \begin{align*}
     \paraa{\pi\times V}(a_s\delta_s\cdot a_t\delta_t)
       & = \paraa{\pi\times V} \parab{ \alpha_s\paraa{\alpha_{s^*}(a_s)a_t) } \delta_{st} }  \\
       \paraa{\pi\times V}(a_s\delta_s) \paraa{\pi\times V}(a_t\delta_t)
       & = \pi(a_s)V_s \pi(a_t)V_t
       = \pi\parab{ \alpha_s\paraa{\alpha_{s^*}(a_s)a_t) } } V_{st} 
  \end{align*}
  The  last step follows from proposition \ref{prop_cov_rep}, item (v).
  Furthermore,	
  \begin{align*}
     \paraa{\pi\times V}\paraa{(a_s\delta_s)^* }
       = \paraa{\pi\times V} \paraa{ \alpha_{s^*}(a_s^*)\delta_{s^*} }
       = \pi\paraa{\alpha_{s^*}(a_{s^*})} V_{s^*} \\
     \parab{ \paraa{\pi\times V}\paraa{a_s\delta_s} }^*
     = \paraa{ \pi(a_s) V_s }^*
     =\pi \paraa{ \alpha_{s^*}(a_s^*)}V_{s^*} 
  \end{align*}
  The last step follows from proposition \ref{prop_cov_rep}, item (vi).
\end{proof}

The following proposition shows that covariant representations of the 
algebra $L(A, S, \alpha)$, which is the basis of the crossed
product algebra (see proposition \ref{cpa_idem_construct}), are
automatically covariant representation of the
crossed product algebra.

\begin{proposition} \label{cov_rep_L1}
 Let $L(A, S, \alpha)$ be the algebra of proposition \ref{cpa_idem_construct}
 (and also of proposition \ref{cpa_construct}). Also let
 $\rho=\pi\times V:L(A, S, \alpha)\rightarrow B(H)$ be a covariant representation.
 Then $\rho(a\delta_r -a\delta_t)=0$ for $r\leq t\in S$ and $a\in A_r$.
\end{proposition}
\begin{proof}
  $r\leq t$ means that $r=qt$ with an idempotent $q\in S$.
  In this case $A_r=A_{qt}=A_q\cap A_t$, see proposition \ref{action_prop_idem},
  and therefore $\pi(a)V_q=\pi(a)$. It follows that
   \begin{align*}
     \rho \paraa{ a\delta_r }
       = \rho \paraa{ a\delta_{qt} }
       = \pi(a) V_{qt}
       = \pi(a) V_q V_t
       = \pi(a) V_t
       = \rho\paraa{a\delta_t}
  \end{align*}
\end{proof}

For crossed product algebras based on algebras with partial 
identities $p_s$, the algebra elements $U_s=p_s\delta_s$
form a representation of the semigroup (see proposition \ref{cpa_rel}).
Covariant representations map the algebra elements $U_s$ 
to the image of the semigroup homomorphims $V_s$, since 
\begin{align*}
  \rho \paraa{U_s} = \pi(p_s)V_s= V_{ss^*}V_s = V_s
\end{align*}
(see propositions \ref{prop_cov_rep_pi}).
The representations of these crossed product algebras on a Hilbert space
are automatically covariant representations.

\begin{proposition} \label{cov_rep_L2}
 Let $L(A, S, \alpha)$ be the algebra of proposition \ref{cpa_construct}, where
 the ideals $A_s$ have partial identities, and let
 $\rho:L(A, S, \alpha)\rightarrow B(H)$ be a Hilbert space representation
 with $\rho(p_s\delta_e - p_{ss^*}\delta_{ss^*})=0$ for $s \in S$.
 Then $\rho$ is a covariant representation.
\end{proposition}
\begin{proof}
  Assume that 
  $\rho(p_s\delta_e - p_{ss^*}\delta_{ss^*})=0$ for $s \in S$.
  We define $\pi\paraa{a}=\rho\paraa{a\delta_e}$, which is obviously
  a $*$-algebra representation on $H$ and
  $V_s=\rho\paraa{p_s\delta_s}$, which
  is a representation of the inverse semigroup $S$ on $H$.
  In particular
  \begin{align*}
     V_s V_t = \rho\paraa{p_s\delta_s \cdot p_t\delta_t}
                = \rho\paraa{p_s \alpha_s{p_{s^*} p_t }\delta_{st} }
                = \rho\paraa{p_s p_{st} \delta_{st} }
                = \rho\paraa{p_{st} \delta_{st} } = V_{st}
  \end{align*}
  Since $V_sV_{s^*}V_s=V_s$ and $V_{s^*}V_sV_{s^*}=V_{s^*}$
  the $V_s$ are partial isometries (proposition \ref{prop_cov_rep}).
    
  Further we have to show that
  $V_s$ has domain $\pi(A_{s^*})H$ and range $\pi(A_{s})H$.
  First let $h \in \pi(A_{s^*})H$, then 
  $h=\pi(a)k=\rho(a\delta_e )$ for some $a\in A_{s^*}$ and
  some $k\in H$. With this
  \begin{align*}
    V_s h = \rho\paraa{p_s\delta_s \cdot a\delta_e } k
            = \rho\paraa{p_s\alpha_s( p_{s^*}a) \delta_s } k
            = \rho\paraa{\alpha_s(a) p_s \delta_s } k
            = \pi \paraa{ \alpha_s(a) } k
  \end{align*}
  Since for every $b\in A_{s}$ there is an $a\in A_{s^*}$
  with $b=\alpha_s(a)$, we see that $V_s$ maps
  $\pi(A_{s^*})H$ onto $\pi(A_{s})H$, i.e.
  the range of $V_s$ is $\pi(A_{s})H$.
  
  Now let $h \in \pi(A_{s^*})H^\perp$, then
  \begin{align*}
     0= <h, \rho\paraa{a\delta_e}k>
         =<\rho\paraa{(a\delta_e)^*} h, k>
         = <\rho\paraa{a^*\delta_e} h, k>
  \end{align*}
  for all $a\in A_{s^*}$ and $k\in H$. Since
  \begin{align*}
    p_s\delta_s\cdot p_{s^*}\delta_e = p_s\alpha_s\paraa{p_{s^*}p_{s^*}} \delta_s = p_s \delta_s
  \end{align*}
  It follows that 
  \begin{align*}
     V_s h = \rho\paraa{p_s \delta_s} h = \rho\paraa{p_s\delta_s} \rho\paraa{ p_{s^*}\delta_e} h = 0
  \end{align*}
  This means that $V_s \pi(A_{s^*})H^\perp = 0$, i.e. the domain of
  $V_s$ is $\pi(A_{s^*})H$.
        
  We also have to show that $V_s\pi(a)V_{s^*} = \pi\paraa{\alpha_s(a)}$
  for all $a\in A_{s^*}$. Now
  \begin{align*}
     V_s\pi(a)V_{s^*} & = \rho\paraa{p_s\delta_s\cdot a\delta_e \cdot p_{s^*} \delta_{s^*} }
                    = \rho\paraa{p_s\delta_s\cdot ap_{s^*} \delta_{s^*} }
                    = \rho\paraa{ p_s \alpha_s\paraa{ p_{s^*}  ap_{s^*} } \delta_{ss^*} } \\
                    & = \rho\paraa{ \alpha_s\paraa{ a } p_s \delta_{ss^*} }
                    = \rho\paraa{ \alpha_s\paraa{ a } p_s \delta_e }
                    = \rho\paraa{ \alpha_s\paraa{ a } \delta_e }
                    = \pi\paraa{ \alpha_s\paraa{ a } }
  \end{align*}
  where we have used that $\rho(p_s\delta_e - p_s \delta_{ss^*})=0$.
\end{proof}

\begin{corollary}
  A representation of the crossed product algebra of
  definition \ref{def_pi_cpa} is a covariant
  representation.
\end{corollary}

\cite{Sieben_1996} shows that the above result is also
true for the crossed product algebra of definition \ref{def_alg_cpa}
in the context of $C^*$-algebras. For the proof, the
partial identities $p_s$ are replaced with approximate identities,
which are always present for closed ideals.

\section{A special covariant representation for function algebras} \label{sec:spec_cov_rep}

\begin{remark} \label{semi_fun_alg}
The function algebras $\mathcal{F}(X)$, $C(X)$ and $C_0(X)$ 
summarized as $\mathcal{\tilde{F}}(X)$ in 
remark \ref{com_fun_alg} are all semiprime, since
every (continuous) function can be written as product of
two (continuous) functions. For example, 
$f(x)=r(x)e^{i\varphi(x)}=\sqrt{r(x)}\cdot \sqrt{r(x)} e^{i\varphi(x)}$.

Thus, when there is a (continuous) action $\alpha$ of the
inverse semigroup $S$ on the set (or topological space) $X$
and $\hat\alpha$ is the induced action according to
proposition \ref{fun_alg_act}, it is possible to form
the crossed product algebra
$\mathcal{\tilde{F}}(X) \rtimes_{\hat\alpha} S$.
\end{remark}

In the following we define a special type of covariant representation for
these type of functions algebras, which are therefore
also representations of the crossed product algebra
$\mathcal{\tilde{F}}(X) \rtimes_{\hat\alpha} S$.

\begin{definition} \label{point_hs}
  Let $\alpha$ be an action of a discrete
  inverse semigroup $S$ on the set $X$ and
  let $x_i \in X, i \in I$ with $I$ a finite index set
  be different points in $X$.  Then the $\alpha$-generated set
  \begin{align*}
    \tilde{X}(\alpha, x_i) = \{\alpha_s(x_i) | s\in S, i\in I, \text{ s. t. } x_i\in X_{s^*}\}
  \end{align*}
  is a discrete subset of $X$. We define the
  Hilbert space $\mathcal{H}(\alpha, x_i)$ as the space
  of square summable functions on the set $\tilde{X}$.
\end{definition}

We write the basis vectors associated with an $y=\alpha_s(x_i) \in \tilde{X}$ 
as $\ket{y}$ and it follows that the inner product of 
two such basis vectors is $<y| y'>=\delta_{y,y'}$. 
With this, every element of $\mathcal{H}(\alpha, x_i)$
is of the form
\begin{align*}
   w = \sum_{y\in \tilde{X}} w_y \ket{y}
\end{align*}
with $w_y$ being complex numbers.
Note that there may be points $y=\alpha_s(x_i)=\alpha_s'(x_i')$,
which can be reached by different sets $(s,i)$ and $(s',i')$.
Such points are solely counted once.

\begin{proposition}  \label{point_hs_rep}
   Let $\alpha$ be an action of an inverse semigroup $S$ on the set $X$,
   let $x_i \in X, i \in I$ with $I$ a finite index set different points in $X$, 
   let $\tilde{X}(\alpha, x_i)$ be the generated point set and
   $\mathcal{H}(\alpha, x_i)$ the Hilbert space such
   as in definition \ref{point_hs}. 
   
   Let $\mathcal{\tilde{F}}(X)$ be one of the function
   algebras $\mathcal{F}(X)$ $C(X)$ or $C_0(X)$,
   then for $s\in S$, $y \in \tilde{X}(\alpha, x_i)$ and $f \in \tilde{\mathcal{F}}(X)$
   \begin{align*}
      V_s \ket{y} = \begin{cases}
  		              \ket{\alpha_s(y)}, \text{ if } y\in X_{s^*} \\
                           0, \text{ otherwise}
                        \end{cases} ,  \qquad
      \pi(f) \ket{y} = f(y) \ket{y}
   \end{align*}
   is a covariant representation on $\mathcal{H}(\alpha, x_i)$
   with respect to the induced action $\hat{\alpha}$
   of $S$ on the algebra $\tilde{\mathcal{F}}(X)$.
\end{proposition} 
\begin{proof}
  We abbreviate $\mathcal{H}(\alpha, x_i)$ with $\mathcal{H}$
  and $\tilde{X}(\alpha, x_i)$ with $\tilde{X}$.

  Let $w\in\mathcal{H}$, then $w=\sum_{y\in\tilde{X}} c_y\ket{y}$
  and $V_s w = \sum_{y\in\tilde{X}\cap X_{t^*}} c_y\ket{\alpha_s(y)}$
  for $s\in S$. Thus, $||V_s w||<||w||$ and $V_s$ is bounded.
  The subspace $\text{ker}(V_s)^\perp$ is the span of all vectors
  $\ket{y}$ with $\alpha_s(y)\neq 0$. Thus, if additionally
   $w\in \text{ker}(V_s)^\perp$, it follows that $||V_s w||=||w||$
   and $V_s$ is a partial isometry.
   
   Let $y\in\tilde{X}$ and $s,t\in S$. 
   Since $\alpha_{st}(y)=\alpha_s \alpha_t(y)$ for $y\in X_{(st)^*}$
   it follows that $V_s V_t \ket{y} = V_{st} \ket{y}$ for $y\in \tilde{X} \cap X_{(st)^*}$.
   According to proposition \ref{action_prop}
   $X_{(st)^*} = \alpha_{t^*}\paraa{X_t\cap X_{s^*}}$.
   Thus if $y\notin X_{(st)^*}$ it follows that $\alpha_t(y)\notin X_t\cap X_{s^*}$
   and therefore $V_s (V_t \ket{y})=0$. Therefore, $V_{st}=V_s V_t$
   on all of $\mathcal{H}$.
  
  When $y\in\tilde{X}\cap X_{s^*}$ and $y'\in\tilde{X}\cap X_{s}$ then
   \begin{align*}  
     <y', V_s y > = <y',\alpha_s(y)>
        =\delta_{y',\alpha_s(y)}=\delta_{\alpha_{s^*}(y'),y}
        =<\alpha_{s^*}(y'),y> =<V_{s^*} y' ,y>
   \end{align*}
   When $y\notin\tilde{X}\cap X_{s^*}$ or $y'\notin\tilde{X}\cap X_{s}$ then
   \begin{align*}  
     <y',V_s y >=0=<V_{s^*} y' ,y>
   \end{align*}
   Therefore  $V_{s^*}=V_s^*$ and together with $V_{st}=V_s V_t$
   we have shown that the mapping $t\rightarrow V_t$
   is a semigroup homomorphism.

   Since evaluation at a point is an algebra homomorphims, 
   the mapping $\pi$ is an algebra representation.

   Let $f\in \tilde{\mathcal{F}}(X_{s^*})$ and $y\in \tilde{X}$. When
   $y\in X_s$ then
   \begin{align*}  
     V_s\pi(f)V_{s^*}\ket{y} & = V_s\pi(f)\ket{\alpha_{s^*}(y)}
      = V_s f\paraa{\alpha_{s^*}(y)} \ket{\alpha_{s^*}(y)}
      = f\paraa{\alpha_{s^*}(y)} \ket{\alpha_s\alpha_{s^*}(y)} \\
      & = f\paraa{\alpha_{s^*}(y)} \ket{y}
      = \pi\paraa{\hat{\alpha}_s(f)} \ket{y}
   \end{align*}
   and when $y \notin X_s$ then
   \begin{align*}  
     V_s\pi(f)V_{s^*}\ket{y} = 0
      = \pi\paraa{\hat{\alpha}_s(f)} \ket{y}
   \end{align*}
   since $\hat{\alpha}_s(f)(X\setminus X_s) = 0$ by definition.

   The domain of $V_s$, i.e. the vectors, which are
   not mapped to $0$ by $V_s$, is the span
   of all vectors $\ket{y}$ with $y \in \tilde{X}\cap X_{s^*}$. 
   Let $v$ be a vector in the domain of $V_s$, then
   \begin{align*}  
     v=\sum_{y \in \tilde{X}\cap X_{s^*}} v_y \ket{y}
   \end{align*}  
   with $v_y\neq 0$. On the other hand, let $w\in \mathcal{H}$ and 
   $f\in \tilde{\mathcal{F}}(X_{s^*})$, such that $f(y)\neq 0$
   for $y \in \tilde{X}\cap X_{s^*}$. We can find such a function
   for all function algebras in consideration. Then
   \begin{align*}  
     \pi(f) w=\sum_{y \in \tilde{X}\cap X_{s^*}} f(y) w_y \ket{y}
   \end{align*} 
   For $v_y=f(y) w_y$ the vectors are equal and we can choose $w$, such that
   $\pi(f)w=v$. It follows that $V_s$ has domain
   $\pi\para{\tilde{\mathcal{F}}(X_{s^*})}\mathcal{H}$.
   
   In the same way one shows that the range of $V_s$,
   which is the span of all vectors $\ket{y}$ with $y \in \tilde{X}\cap X_s$
   is equal to $\pi\paraa{\tilde{\mathcal{F}}(X_{s})}\mathcal{H}$.
\end{proof}

\begin{corollary} \label{point_hs_cpa_rep}
  In the nomenclature of proposition \ref{point_hs_rep}
  \begin{align*}
     \sum_{s\in S}^\text{finite} f_s \delta_s \mapsto \sum_{s\in S}^\text{finite} \pi\paraa{f_s}V_s
  \end{align*}
  is a representation of the crossed product algebra 
  $\mathcal{\tilde{F}}(X) \rtimes_{\alpha} S$ on the
  Hilbert space $H$
  (see proposition \ref{cov_rep_is_alg_rep}).
  
  In the case, when $\mathcal{\tilde{F}}(X)=\mathcal{F}(X)$
  is the complete function algebra, which contains 
  partial identities $p_s=\text{id}_X$ for all ideals
  $\mathcal{F}(X)$ then
    \begin{align*}
     U_s = p_s \delta_s \mapsto V_s
  \end{align*}
\end{corollary}

Note that the partial identity $p_s$ for the ideals $\mathcal{F}(X)$
can be seen as a limit of continuous bump functions,
which approximate the identity function inside the set
$X_s$. The limit function $p_s$ is not continuous,
since there is a jump from $1$ to $0$ at the border of $X_s$.

\section{Partial group actions} \label{sec:pga} 

\cite{Sieben_1996} shows that a partial group action generates
an inverse semigroup. In \cite{Exel_1998} it is shown
that every group is associated with a universal inverse semigroup and
that the partial actions of the group on a set are in one-to-one
correspondence with the actions of the inverse semigroup on that set. 
We shortly explore this relationship, since in the first part we use the inverse
semigroup generated by a partial group action of $\integers$ on an
interval of $\reals$. 

A partial action of a group on a set is defined
as follows (see for example \cite{Doku_2005}):

\begin{definition} \label{pga}
 Let $G$ be a discrete group and  $X$ a set. A \emph{partial action}
 $\alpha$ of $G$ on $X$ is a collection of subsets $X_g \subseteq X$ $(g\in G)$
 and partial bijections $\alpha_g:X_{g^{-1}}\rightarrow X_g$, such that
  \begin{enumerate}[(i)]
    \item $X_e=X$,
    \item $X_{(gh)^{-1}} \supset \alpha_h^{-1}\paraa{X_h \cap X_{g^{-1}} }$,
    \item $\alpha_{gh}|\alpha_h^{-1}\paraa{X_h \cap X_{g^{-1}} }= \alpha_g\alpha_h$
  \end{enumerate}
  for all $g,h\in G$.
\end{definition}

Note that the range of $\alpha_g\alpha_h$ is $\alpha_h^{-1}\paraa{X_h \cap X_{g^{-1}} }$
and therefore items (ii) and (iii) mean that the mapping $\alpha_{gh}$
extends the mapping $\alpha_g\alpha_h$.
  
In general, a group action on a set restricted to a subset results in a partial group action.
(The following proofs are adapted from \cite{Sieben_1996}.)

\begin{proposition} \label{restricted_action}
Let $\beta$ be a group action of the group $G$ on the set $Y$. Then the action
$\alpha$ defined by the action $\beta$ restricted to the subset $X\subset Y$
with $X_g=X \cap \beta_g(X)$ and $\alpha_g=\beta_g|X_{g^{-1}}$, i.e.
$\beta$ restricted to $X_{g^{-1}}$ $(g\in G)$ is a partial action on $X$.
\end{proposition}
\begin{proof}
   $X_e=X \cap \beta_e(X)=X$ and (i) from above definition is fulfilled.

   For $g\in G$, $\alpha_g$ is a partial bijection with
   inverse $\alpha_{g^{-1}}$, since 
   $\text{ran}(\alpha_g)=\text{dom}(\alpha_{g^{-1}})$
   and
   \begin{align*}
     \alpha_{g^{-1}}\paraa{X_g) } =
     \alpha_{g^{-1}}(X\cap \beta_g{X})=
     \beta_{g^{-1}}\paraa{X\cap \beta_g(X) } =
     \beta_{g^{-1}}(X) \cap X=X_{g^{-1}}
   \end{align*}

  For $g,h\in G$, by definition of the concatenation of
  partial bijections, the domain of  $\alpha_g\alpha_h$ is 
   \begin{align*}
     \text{dom}\paraa{ \alpha_g\alpha_h } & = \alpha_{h^{-1}} \paraa{\text{ran}(\alpha_h)\cap \text{dom}(\alpha_g) } \\
     & = \alpha_{h^{-1}} \paraa{X\cap \beta_h(X) \cap \beta_{g^{-1}}(X) } \\
     & = \beta_{h^{-1}} \paraa{\text{dom}(\alpha_{h^{-1}})  \cap  X \cap \beta_h(X) \cap \beta_{g^{-1}}(X) } \\
     & = \beta_{h^{-1}}(X) \cap X \cap \beta_{h^{-1}g^{-1}}(X)
   \end{align*}
   On the other hand
   \begin{align*}
     \text{dom}\parab{ \alpha_{gh}|\alpha_h^{-1}\paraa{ X_{g^{-1}} } } & =
               X \cap \beta_{(gh)^{-1}}(X) \cap \alpha_{h^{-1}}\paraa{ X \cap \beta_{g^{-1}}(X) }  \\
      & = X \cap \beta_{(gh)^{-1}}(X) \cap \beta_{h^{-1}}\paraa{ \text{dom}(\alpha_{h^{-1}})\cap X \cap \beta_{g^{-1}}(X) } \\
      & = X \cap \beta_{(gh)^{-1}}(X) \cap  \beta_{h^{-1}}(X) 
   \end{align*}
   thus, since $\beta_g\beta_h=\beta_{gh}$, $\alpha_{gh}$
   extends $\alpha_g\alpha_h$.
\end{proof}

A further example is the partial
group action generated by a single partial bijection.
\begin{proposition} \label{single_bij_action}
Let $\alpha:X_{-1}\rightarrow X_1$
be a partial bijection of the set $X$. Then the bijections
$\alpha_n=\alpha^n:X_{-n}\rightarrow X_n$
with $n\in \integers$ are a partial group action
of $\integers$ on $X$.
 Let $X_\cap=X_{-1} \cap X_1$. Then the
 domain and range of $\alpha_n$ are
   \begin{align*}
      X_{-n} &=\alpha^{-(n-1)}\paraa{X_\cap} \cap \cdots \cap \alpha^{-1} \paraa{X_\cap} \\
      X_{n} &=\alpha^{n-1} \paraa{X_\cap} \cap \cdots \cap \alpha \paraa{X_\cap}
   \end{align*}
   for $n=2, 3, \dots$.
\end{proposition}
\begin{proof}
  Since $\alpha_0=\alpha^0=\text{id}_X$ and $X_0=X$,
  item (i) of definition \ref{pga} is fulfilled.
  Items (ii) and (iii) follow automatically, since 
  $\alpha_n\alpha_m=\alpha^n\alpha^m=\alpha^{n+m}=\alpha_{n+m}$.
  
  The sets for domain and range of $\alpha_n$ can be derived by induction by
  using that
     \begin{align*}
        X_n = \alpha\paraa{ X_{n-1} \cap X_{-1} }
          = \alpha\paraa{ X_{n-1} \cap X_1 \cap X_{-1} }  
          = \alpha\paraa{ X_{n-1} \cap X_\cap }  
          = \alpha\paraa{ X_{n-1} } \cap \alpha\paraa{ X_{\cap} }
     \end{align*}
     and that $X_{-n}$ is the range of $\alpha_{-n}$.
\end{proof}

In particular, the domain and range for $n+1$ are contained
in the domain and range for $n$, i.e. for $n>0$, 
$X_{n+1} \subset X_{n}$ and $X_{-n-1} \subset X_{-n}$.

\begin{proposition} \label{pga_prop}
If $\alpha$ is a partial group action of $G$ on $X$, then
   \begin{enumerate}[(i)]
    \item $\alpha_e$ is the identity map on $X$
    \item $\alpha_{g^{-1}}=\alpha_g^{-1}$ 
    \item $\alpha_g(X_{g^{-1}} \cap X_h) = X_g \cap X_{gh}$
    \item $\alpha_{gh}\alpha_{h^{-1}}=\alpha_{g}\alpha_{h}\alpha_{h^{-1}}$ and 
            $\alpha_{g^{-1}}\alpha_{gh}=\alpha_{g^{-1}}\alpha_{g}\alpha_h$ 
    \item $\alpha_{g}=\alpha_{g}\alpha_{g^{-1}}\alpha_{g}$ and 
            $\alpha_{g^{-1}}=\alpha_{g^{-1}}\alpha_{g}\alpha_{g^{-1}}$ 
  \end{enumerate}
  for all $g,h\in G$.
  \end{proposition}
\begin{proof}
 (i) $\alpha_{ee}$ extends $\alpha_e\alpha_e$ on $\alpha_e^{-1}(X_e=X)$, 
     therefore
   \begin{align*}
     \text{id}=\alpha_e\alpha_e^{-1}=\alpha_{ee}\alpha_e^{-1}
       =\alpha_e\alpha_e\alpha_e^{-1} = \alpha_e
   \end{align*}
 
 (ii) $\alpha_{gg^{-1}}$ extends $\alpha_g\alpha_{g^{-1}}$
          on $\alpha_{g^{-1}}^{-1}(X_{g^{-1}})=X_g$, therefore
   \begin{align*}
     \alpha_g\alpha_{g^{-1}} =\alpha_{gg^{-1}}|X_g
     =\alpha_{e}|X_g=\text{id}|X_g 
   \end{align*}
   
 (iii) By item (ii) of definition \ref{pga} and (ii) above
     \begin{align*}
        \alpha_g(X_{g^{-1}} \cap X_h)=\alpha_{g^{-1}}^{-1}(X_{g^{-1}} \cap X_h) \subset X_{gh}
     \end{align*}
    Since $X_g$ is the range of $\alpha_g$
    \begin{align*}
       \alpha_g(X_{g^{-1}} \cap X_h)\subset X_g\cap X_{gh}
    \end{align*}       
    
    Replacing $g$ with $h^{-1}$ and $h$ with $hg$ results in
    $\alpha_{h^{-1}}(X_h \cap X_{hg})\subset X_{h^{-1}}\cap X_g$ and by 
    applying $\alpha_h$, 
    \begin{align*}
       X_h \cap X_{hg} \subset \alpha_h(X_{h^{-1}}\cap X_g)
    \end{align*}
    
 (iv) By definition, the concatenation of the partial bijections
      $\alpha_{gh}$ and $\alpha_{h^{-1}}$ is between the sets
      \begin{align*}
          \alpha_{gh}\alpha_{h^{-1}}: X_h \rightarrow X_{h^{-1}}\cap X_{h^{-1}g^{-1}} \rightarrow X_{gh}
      \end{align*}
      (where the outer sets are further restricted).
      By (iii), the set in the middle is 
      $\alpha_{h^{-1}}(X_h\cap X_{g^{-1}}) \subset \alpha_{h^{-1}}(X_{g^{-1}})$
      and on this set, $\alpha_{gh}$ extends $\alpha_{g}\alpha_{h}$.
      The second equation follows by inverting the first equation, applying (ii) and exchanging $g$ and $h$ with $h^{-1}$ and
       $g^{-1}$.
       
  (v) This is (iv) with $h=g^{-1}$.
\end{proof}

\begin{proposition} \label{pga_prop_idem}
  Let $\alpha$ be a partial group action of $G$ on $X$. Then
  $\epsilon_g = \alpha_g \alpha_g^{-1}$ is an idempotent
  with
   \begin{enumerate}[(i)]
    \item $\epsilon_g \epsilon_h = \epsilon_h \epsilon_g$
    \item $\alpha_h \epsilon_g = \epsilon_{hg} \alpha_h$ 
  \end{enumerate}
  for all $g,h\in G$.
  \end{proposition}
\begin{proof}
   $\epsilon_g$ is an idempotent since
   \begin{align*}
     \epsilon_g^2  = \alpha_g \alpha_g^{-1} \alpha_g \alpha_g^{-1} = \alpha_g \alpha_g^{-1} = \epsilon_g
   \end{align*}
   
   (i) and (ii) follow by direct calculation:
   \begin{align*}
     \epsilon_g \epsilon_h & = \alpha_g \alpha_g^{-1} \alpha_h \alpha_{h^{-1}}
        = \alpha_g \alpha_{g^{-1} h } \alpha_h^{-1}
        = \alpha_g \alpha_{g^{-1} h } \alpha_{h^{-1} g } \alpha_{g^{-1} h }\alpha_h^{-1} \\
      &  = \alpha_{ g g^{-1} h } \alpha_{h^{-1} g } \alpha_{g^{-1} h h^{-1} }
         = \alpha_{h} \alpha_{h^{-1} g } \alpha_{g^{-1}}
        = \alpha_{h} \alpha_{h^{-1}} \alpha_{g} \alpha_{g^{-1}}
        = \epsilon_h \epsilon_g 
   \end{align*}
   \begin{align*}
     \alpha_g \epsilon_h & = \alpha_g \alpha_h \alpha_{h^{-1}}
     = \alpha_g \alpha_{g^-1} \alpha_g \alpha_h \alpha_{h^{-1}}
     = \alpha_g \epsilon_{g^-1} \epsilon_{h^{-1}}
     = \alpha_g \epsilon_{h^{-1}} \epsilon_{g^-1} \\
     &= \alpha_g \alpha_h \alpha_{h^{-1}} \alpha_{g^-1} \alpha_g
     = \alpha_{gh} \alpha_{h^{-1}} \alpha_{g^-1} \alpha_g
     = \alpha_{gh} \alpha_{h^{-1}g^{-1}} \alpha_g
     = \epsilon_{gh} \alpha_g
   \end{align*}
\end{proof}
(The proof is similar to the proof of proposition 2.2 of \cite{Buss_2011}.)

As already mentioned, a partial group action gives rise to an
inverse semigroup. We show this with the aid of two lemmas.

\begin{lemma} \label{pga_lemma1}
   Let $\alpha$ be a partial group action of the group $G$ on the set $X$,
   let $\epsilon_g = \alpha_g \alpha_g^{-1}$, $\epsilon_h = \alpha_h \alpha_h^{-1}$
   and $t=\alpha_{g_1}\dots\alpha_{g_n}$ for $g, h, g_1, \dots, g_n \in G$. Then
   \begin{align*}
     t \epsilon_h t^{-1} \epsilon_g = \epsilon_g t \epsilon_h t^{-1} 
   \end{align*}
\end{lemma}
\begin{proof}
  We move the idempotent $\epsilon_g$ from the right through $t^{-1}$, where
  according to proposition \ref{pga_prop_idem}, (ii) it becomes $\epsilon_{g_n^{-1}\dots g_1^{-1} g}$.
  This idempotent commutes with $\epsilon_h$ and we can move it through
  $t$, where it becomes $\epsilon_{g_1 \dots g_n g_n^{-1}\dots g_1^{-1} g}=\epsilon_{g}$.
\end{proof}

\begin{lemma} \label{pga_lemma2}
   Let $\alpha$ be a partial group action of the group $G$ on the set $X$. Then
   \begin{align*}
       \alpha_h\paraa{X_{h^{-1}} \cap X_{g_1}\cap & X_{g_1 g_2 } \cap \cdots \cap X_{g_1\dots g_n} } 
         = X_h \cap X_{hg_1}\cap X_{hg_1 g_2 } \cap \cdots \cap X_{hg_1\dots g_n} 
   \end{align*}
   for $h, g_1, \dots, g_n \in G$.
\end{lemma}
\begin{proof}
  Using proposition \ref{pga_prop}, (iii), we calculate
  \begin{align*}
     \alpha_h\paraa{X_{h^{-1}} \cap X_{g_1}\cap & X_{g_1 g_2 } \cap \cdots \cap X_{g_1\dots g_n} } \\
         &= \alpha_h\paraa{X_{h^{-1}}\cap X_{g_1}} \cap \alpha_h\paraa{X_{h^{-1}} \cap  X_{g_1 g_2 } } \cap \cdots \cap
          \alpha_h\paraa{X_{h^{-1}} \cap X_{g_1\dots g_n} }   \\
         &= X_h \cap X_{h g_1} \cap X_h \cap X_{h g_1 g_2 }  \cap \cdots \cap
          X_h  \cap X_{h  g_1\dots g_n} \\
         &= X_h \cap X_{h g_1} \cap X_{h g_1 g_2 }  \cap \cdots \cap X_{h  g_1\dots g_n}           
  \end{align*} 
  where we have used that $\alpha_h(A\cap B)=\alpha_h(A) \cap \alpha_h(B)$ for two subsets $A$ and $B$
  of $X_{h^{-1}}$ since $\alpha_h$ is a bijection.
\end{proof}

\begin{proposition} \label{pga_semigroup}
Let $\alpha$ be a partial group action of the group $G$ on the set $X$. 
Then the partial bijections $\alpha_g$ with $g\in G$ generate a unital 
inverse subsemigroup $S_\alpha(G)$ of the inverse
monoid $I(X)$ of $X$.

For an element $s=\alpha_{g_1}\dots\alpha_{g_n}$ of $S_\alpha(G)$, the
domain and range are
  \begin{align*}
    s: X_{g_n^{-1}}\cap X_{g_n^{-1}g_{n-1}^{-1} } \cap \cdots \cap X_{g_n^{-1}\dots g_{1}^{-1} } \rightarrow
   X_{g_1}\cap X_{g_1 g_2 } \cap \cdots \cap X_{g_1\dots g_n}
  \end{align*}
\end{proposition}
\begin{proof}
  According to items (i) of proposition \ref{pga_prop} the unit
  is $\alpha_e$. 
  
  We show that $ss^{-1} s=s$ for $s=\alpha_{g_1}\dots\alpha_{g_n}$ 
  by induction.
    
  The case $n=1$ is item (v) of proposition \ref{pga_prop}.
  
  For the general case, we set $t= \alpha_{g_2}\dots\alpha_{g_{n-1}}$. Then by induction
  \begin{align*}
    s s^{-1} s & = \alpha_{g_1} t \alpha_{g_n} \alpha_{g_n^{-1}}  t^{-1} \alpha_{g_1^{-1}} \alpha_{g_1} t \alpha_{g_n} \\
               & = \alpha_{g_1} \alpha_{g_1^{-1}}  \alpha_{g_1} t \alpha_{g_n} \alpha_{g_n^{-1}}  t^{-1} t \alpha_{g_n} \\
               & = \alpha_{g_1} t t^{-1} t \alpha_{g_n} \alpha_{g_n^{-1}}  \alpha_{g_n} \\
               & = \alpha_{g_1} t t^{-1} t \alpha_{g_n} \\
               & = \alpha_{g_1} t \alpha_{g_n} \\
               & = s
  \end{align*}  
  Here we have used proposition \ref{pga_prop_idem} and lemma \ref{pga_lemma1}.
  In the same way, it can be shown that $s^*s s^* =s^*$.
  
  We also show the statement about the domain and range by induction. For $n=1$, the statement
  is true by definition, and by lemma \ref{pga_lemma2}
   \begin{align*}
     \text{dom} \paraa{ \alpha_{g_1}\dots \alpha_{g_n} } 
         & = \alpha_{g_n}^{-1}\paraa{ X_{g_n} \cap \text{dom} \paraa{ \alpha_{g_1}\dots \alpha_{g_{n-1} } } } \\
         & = \alpha_{g_n}^{-1}\paraa{ X_{g_n}  \cap X_{g_{n-1}^{-1}}\cap \cdots \cap X_{g_{n-1}^{-1}\dots g_1^{-1}} } \\
         & =  X_{g_n^{-1}} \cap X_{g_n^{-1}g_{n-1}^{-1}}\cap \cdots \cap X_{{g_n}^{-1}g_{n-1}^{-1}\dots g_1^{-1}} 
  \end{align*}
  The range of $\alpha_{g_1}\dots \alpha_{g_n}$ is the domain of
  $\alpha_{g_n^{-1}}\dots \alpha_{g_1^{-1}} $.
\end{proof}

The inverse semigroup $S_\alpha(G)$ is composed of partial bijections
of the set $X$, which therefore acts naturally on
the set $X$. 

Item (iv) of proposition \ref{pga_prop} provide
relations of the inverse semigroup $S_\alpha(G)$. With the aid
of these relations, it can be shown that every element $s\in S_\alpha(G)$
is a product of idepotents with one element $\alpha_g$
(see \cite{Buss_2011}, Proposition 2.5). 

\begin{proposition} \label{pga_sg_elem}
  Let $\alpha$ be a partial group action of the group $G$ on the set $X$. Then
  every element $s$ of $S_\alpha(G)$ is of the form
  \begin{align*}
    s = \epsilon_{g_1} \cdots \epsilon_{g_n} \alpha_g
   \end{align*}
   for $g_1,\dots,g_n, g \in G$ and
   $\epsilon_{g_i}=\alpha_{g_i}\alpha_{g_i^{-1}}$ idempotents.
   The range of $s$ is
   \begin{align*}
     X_s = X_{\alpha_{g_1}} \cap X_{\alpha_{g_n}} \cap X_{\alpha_{g}}
   \end{align*}
\end{proposition} 
\begin{proof}
  Let $s=\alpha_{h_1}\cdots \alpha_{h_n} \in S_\alpha(G)$.
  For $n=0$ the assertion is trivial. Assume that
  $\alpha_{h_2}\cdots \alpha_{h_n} = \epsilon_{g_1} \cdots \epsilon_{g_m} \alpha_g$
  for some $g_1,\dots,g_m,g\in G$ . Then
  \begin{align*}
    s & = \alpha_{h_1}\alpha_{h_2}\cdots \alpha_{h_n} 
        = \alpha_{h_1} \epsilon_{g_1} \cdots \epsilon_{g_m} \alpha_g \\
       & = \epsilon_{h_1 g_1} \cdots \epsilon_{h_1 g_m} \alpha_{h_1} \alpha_g 
        = \epsilon_{h_1 g_1} \cdots \epsilon_{h_1 g_m} \alpha_{h_1} \alpha_{h_1^{-1}} \alpha_{h_1}\alpha_g \\
       & = \epsilon_{h_1 g_1} \cdots \epsilon_{h_1 g_m} \alpha_{h_1} \alpha_{h_1^{-1}} \alpha_{h_1 g} 
        = \epsilon_{h_1 g_1} \cdots \epsilon_{h_1 g_m} \epsilon_{h_1} \alpha_{h_1 g}
  \end{align*}
  by propositions \ref{pga_prop} and \ref{pga_prop_idem}. Thus
  the assertion follows by induction.
  
  The formula for the range of $X_s$ follows from
  proposition \ref{action_prop_idem}, since the range
  (as well as the domain) of the projector $\epsilon_g$ are $X_g$.
\end{proof}

In the end of this section, we investigate the
inverse semigroup  $S_\alpha(\integers)$ generated
by a single bijection $\alpha:X_{-1}\rightarrow X_1$
of a set $X$ (see proposition \ref{single_bij_action}).

\begin{proposition} \label{single_bij_sg}
Let $\alpha:X_{-1}\rightarrow X_1$
be a partial bijection of the set $X$ and 
$\alpha_n=\alpha^n:X_{-n}\rightarrow X_n$
be the partial group action of $\integers$ on $X$.
Then the elements of the inverse
semigroup $S_\alpha(\integers)$ are 
 \begin{align*}
    s = \epsilon_{n_+}\epsilon_{n_-} \alpha_m 
 \end{align*}
 for $n_+ \geq 0$, $n_- \leq 0$ and $m$ an integer.
\end{proposition}
\begin{proof}
   For $n>0$ the sets $X_n$  have the property that $X_n \subset X_{n-1}$
   and $X_{-n} \subset X_{-(n-1)}$. Thus, a product of idempotents
   $\epsilon_{n_i} = \alpha_{n_i} \alpha_{-n_i}$ with the $n$ either all less than
   $0$ or all greater than $0$ is equal to the idempotent of highest $n$ or smallest $n$, 
   respectively. 
   Thus every element $s=\epsilon_{n_1}\cdots \epsilon_{n_N} \alpha_m$ 
   of $S_\alpha(\integers)$ reduces to
   $s=\epsilon_{n_+}\epsilon_{n_-} \alpha_m$,
   where $n_+$ is the greatest and $n_-$ is the smallest value of $n_1$ to $n_N$.
\end{proof}

\section{Crossed product algebras for partial group actions} \label{sec:CPA_PGA}

We have just have seen that the action
$\alpha$ creates a semigroup $S_\alpha(G)$
(see proposition \ref{pga_semigroup}). We use this for defining the
crossed product algebra for partial group actions, such as in
 \cite{Sieben_1996}.

\begin{definition} \label{def_pga_alg}
  A partial group action $\alpha$ of the 
  group $G$ on the algebra $A$ is a partial
  group action of $G$ on the set $A$, where the partial
  bijections $\alpha_g: A_{g^{-1}}\rightarrow A_g$ are
  partial automorphims between ideals $A_{g^{-1}}$
  and $A_g$ of $A$.
\end{definition}

\begin{definition} \label{def_pa_cpa}
  Let $\alpha$ be a partial action of the group 
  $G$ on the $*$-algebra $A$, where the domains and
  ranges of the elements of the inverse semigroup 
  $S_\alpha(G)$ are non-degenerate or idempotent ideals.
  Then the \emph{algebraic crossed product} of $A$ and $G$
  by the action $\alpha$ is 
   \begin{align*} 
      A \rtimes_{\alpha} G = A \rtimes_{\alpha} S_\alpha(G)
  \end{align*}
\end{definition}

Following proposition shows that the crossed product algebra for
a partial group action can be reduced to discrete functions
on the group $G$ alone. Here, it is important that the
crossed product algebra is a factor algebra.
\begin{proposition} \label{pa_cpa_prop}
  Let $\alpha$ be a partial action of the group 
  $G$ on the $*$-algebra $A$, according to definition \ref{def_pa_cpa}.
  Then the elements $\phi$
  of $A \rtimes_{\alpha} G$ are of the form
   \begin{align*} 
      \phi = \sum_{g\in G}^{\text{finite}} a_g \delta_g
  \end{align*}
  where $a_g\in A_g$ and we have
  abbreviated $\delta_g=\delta_{\alpha_g}$
  For $g, h \in G$ the algebra multiplication
  and involution translates into
  \begin{align*} 
    a_g\delta_g \cdot  a_h\delta_h = \alpha_g \paraa{ \alpha_{g^{-1}}\paraa{a_g}a_h } \delta_{gh} 
  \end{align*}
    \begin{align*} 
      (a_g\delta_g)^* = \alpha_{g^{-1}}(a_g^*)\delta_{g^{-1}}
  \end{align*}
  where $a_g\in A_g$ and $a_h\in A_h$.
\end{proposition}
\begin{proof} 
   Let $\phi \in A \rtimes_{\alpha} G$. Then
   \begin{align*}
      \phi = \sum_{s\in S_\alpha(G)}^{\text{finite}} a_s \delta_s
   \end{align*}
   where $a_s \in A_s$.
   Every element $s\in S_\alpha(G)$ is of the form
   $ s = \epsilon_{g_1} \cdots \epsilon_{g_n} \alpha_h $
   and $A_s \subset A_g$ (see proposition \ref{pga_sg_elem}).
   Therefore $a_s \in A_g$, $s \leq \alpha_h$ and 
    $a_s \delta_s = a_s \delta_g$. For each $g$, we set 
    $a_g = \sum_{s \leq g} a_s$ and it follows that
   \begin{align*} 
      \phi = \sum_{g\in G}^{\text{finite}} a_g \delta_g
  \end{align*}    
    For the multiplication, we have to show that
   $\delta_{\alpha_g\alpha_h}$ can be replaced with $\delta_{\alpha_{gh}}$.
   Now 
   \begin{align*} 
     \alpha_g \alpha_h = \alpha_g \alpha_g^{-1}\alpha_g \alpha_h
      = \alpha_g \alpha_g^{-1}\alpha_{gh } 
   \end{align*}
   since the $\alpha_g$ are elements of the inverse
   semigroup $S_\alpha(G)$. The last step follows from
   item (v) of proposition \ref{pga_prop}. Thus
   $\alpha_g \alpha_h \leq \alpha_{gh }$ and
   \begin{align*} 
     a_g\delta_g \cdot  a_h\delta_h 
        = a_g\delta_{\alpha_g} \cdot  a_h\delta_{\alpha_h}
        = \alpha_g \paraa{ \alpha_{g^{-1}}\paraa{a_g}a_h } \delta_{\alpha_g \alpha_h} 
        = \alpha_g \paraa{ \alpha_{g^{-1}}\paraa{a_g}a_h } \delta_{\alpha_{gh} } 
  \end{align*}

    The algebra involution follows directly from
  the definition.
\end{proof}

 The following crossed product algebra is the basis for all fuzzy spaces defined
 in the first part. It is the crossed product algebra formed in a canonical
 way from one partial bijection $\alpha$ of a set $X$.

\begin{proposition} \label{single_bij_cpa}
  Let $\alpha_n=\alpha^n, n\in\integers$
  be the partial action of the group $\integers$
  on the set $X$ generated by a
  single partial bijection $\alpha$ of $X$, such as defined
  in \ref{single_bij_action}. 
  Further, let $\beta_{\alpha^n}=\alpha^n$ be the action
  of  $S_\alpha(\integers)$ on $X$,
  and $\hat\beta$ the induced action on
  one of the function algebras
  $\mathcal{\tilde{F}}(X)=\mathcal{F}(X), C(X), C_0(X)$
  (remark \ref{com_fun_alg}). Then the crossed product algebra
 \begin{align*} 
       \mathcal{\tilde{F}}(X) \rtimes_{\alpha} \integers = 
           \mathcal{\tilde{F}}(X) \rtimes_{\hat\beta} S_\alpha(\integers)
  \end{align*}  
  can be formed. The elements $\phi$
  of $\mathcal{\tilde{F}}(X) \rtimes_{\alpha} \integers$ are of the form
   \begin{align*} 
      \phi = \sum_{n\in \integers}^{\text{finite}} f_n \delta_n
  \end{align*}
  where $f_n\in \mathcal{\tilde{F}}(X_n)$. The algebra multiplication
  and involution translates into
  \begin{align*} 
    f_n \delta_n \cdot  f_m \delta_m = \paraa{ f_n ( f_m |_{X_{-n}} \circ \alpha^{-n} ) } \delta_{n+m} 
  \end{align*}
    \begin{align*} 
      (f_n\delta_n)^* = \paraa{\overline{f_n}\circ \alpha^{n}}\delta_{-n}
  \end{align*}
  where $f_n\in \mathcal{\tilde{F}}(X_n)$ and $f_m\in \mathcal{\tilde{F}}(X_m)$.
\end{proposition}
\begin{proof}
  The form of the elements and the involution immediately follows
  from proposition \ref{pa_cpa_prop} by inserting the induced action
  defined in proposition \ref{fun_alg_act}.
  
  For the algebra multiplication this results in 
  \begin{align*} 
    f_n \delta_n \cdot  f_m \delta_m = \parab{ \paraa{ (f_n\circ \alpha^{n}) f_m }\circ \alpha^{-n} } \delta_{n+m} 
      =  \paraa{ f_n ( f_m |_{X_{-n}} \circ \alpha^{-n} ) } \delta_{n+m} 
  \end{align*}
  The last step follows since $f_n\circ \alpha^{n} \in \mathcal{\tilde{F}}(X_{-n})$
  and we can replace $(f_n\circ \alpha^{-n}) f_m$ with $f_n\circ \alpha^{-n}) (f_m |_{X_{-n}})$,
  i.e. we can restrict $f_m$ to $X_{-n}$, such that it is in the range of $\alpha^{-n}$.
\end{proof}
Note that the restriction to $X_{-n}$ also can be achieved by 
multiplication with the support function of $X_{-n}$, which is present,
when the function algebra is the algebra of all functions $\mathcal{F}(X)$.
This is also the result of the two following propositions.

In particular, when we further assume that the algebra $A$ has partial identities,
we further can define the crossed product algebra based on elements
$U_g$ and their relations.
\begin{proposition} \label{pa_cpa_pi}
    Let $\alpha$ be a partial action of the group 
   $G$ on the $*$-algebra $A$, according to definition \ref{def_pa_cpa},
   where additionally, the domains and
  ranges of the elements of the inverse semigroup 
  $S_\alpha(G)$ have partial identities.
   Then every element $\phi$
    of $A \rtimes_{\alpha} G$ is of the form
   \begin{align*} 
      \phi =\sum_{g\in G}^{\text{finite}} a_g U_g
   \end{align*}  
   where $a_g \in A$ and $U_g=p_{\alpha_g} \delta_{\alpha_g}$.
   The elements $U_g$ fulfil 
   \begin{align*}   
      U_g U_h = p_{\alpha_g} U_{gh}, \qquad \paraa{U_g}^*=U_{g^{-1}},
     \qquad       U_g a  = \alpha_g\paraa{ p_{\alpha_g^{-1}} a } U_g
   \end{align*}  
   for $g,h\in G$ and $a \in A$.
\end{proposition}
\begin{proof}
  According to proposition \ref{pa_cpa_prop}, the elements
  of $A \rtimes_{\alpha} G$ are of the form
   $ \phi = \sum_{g\in G}^{\text{finite}} a_g \delta_g$. Since
   $a_g = a_g p_g$ for every $a_g\in A_g$ and $a p_g \in A_g$
   for $a\in A$, it follows that $ \phi = \sum_{g\in G}^{\text{finite}} a_g U_g$
   for $a_g\in A$.
   
   For showing the algebra relations, we note that
    \begin{align*} 
         p_{\alpha_g\alpha_h}
          & =p_{\alpha_g\alpha_{g^{-1}}\alpha_g\alpha_h }
          =p_{\alpha_g\alpha_{g^{-1}}\alpha_{gh} }
          =\alpha_g\paraa{p_{\alpha_{g^{-1}}} \alpha_{g^{-1}} \paraa{ p_{\alpha_g} p_{\alpha_{gh}} } } \\
          & = \alpha_g\paraa{\alpha_{g^{-1}} \paraa{ p_{\alpha_g} p_{\alpha_{gh}} } }
          = p_{\alpha_g} p_{\alpha_{gh}}
    \end{align*} 
    (which is equivalent to item (iii) of proposition \ref{pga_prop}.)
    Together with the aid of item (v) of proposition \ref{cpa_rel}
    this results in 
    \begin{align*} 
      U_g U_h = U_{\alpha_g} U_{\alpha_h}
                 =U_{\alpha_g\alpha_h}
                 = p_{\alpha_g\alpha_h} \delta_{\alpha_g\alpha_h} 
                 = p_g p_{\alpha_{gh}} \delta_{\alpha_{gh} }
                 = p_g U_{gh}
    \end{align*} 

   The involution is
   \begin{align*}
      \paraa{U_g}^* = \paraa{U_{\alpha_g}}^*  = U_{(\alpha_g)^*} 
             =U_{\alpha_{g^{-1}}} = U_{g^{-1}}
   \end{align*}        
   by item (iv) of proposition \ref{cpa_rel} and
   item (ii) of proposition \ref{pga_prop}.
   
   The last assertion follows by applying definitions
   \begin{align*}
      U_g a  &= p_{\alpha_g} \delta_{\alpha_g} \cdot a \delta_{\alpha_e}
      = p_{\alpha_g} \alpha_g\paraa{ p_{\alpha_g^{-1}} a } \delta_{\alpha_g\alpha_e} \\
      & = \alpha_g\paraa{ p_{\alpha_g^{-1}} a } p_{\alpha_g}  \delta_{\alpha_g}  
            = \alpha_g\paraa{ p_{\alpha_g^{-1}} a } U_g
   \end{align*}
\end{proof}

\begin{remark}
  Since $\alpha$ is a partial group action and not an (ordinary) group action,
  the elements $U_g$ do not from a representation of the group $G$,
  but reflect properties of the partial group action. This is due to
  the fact that the product of the actions $\alpha_g$ and $\alpha_h$
  of two group elements $g, h\in G$ is not the
   action $\alpha_{gh}$ of the product of the group elements.
  In particular, the $U_g$ are not invertible but fulfil
   \begin{align*}
      U_g U_{g^{-1}} = p_{\alpha_g}, \qquad U_{g^{-1}} U_g = p_{\alpha_{g^{-1}}}
   \end{align*}
   In covariant representations these elements will be mapped to
   partial isometries.
\end{remark}
  
\begin{proposition} \label{single_bij_cpa_pi}
  When we assume further in proposition \ref{single_bij_cpa}
  that $\mathcal{\tilde{F}}(X)=\mathcal{F}(X)$ is the
  full function algebra and the ideals $\mathcal{F}(X_n)$ contain
  the partial identities $p_n$ (which are the functions, which
  are $1$ on $X_n$ and $0$ on $X\setminus X_n$), then
  every element $\phi$
  of $\mathcal{\tilde{F}}(X) \rtimes_{\alpha} \integers$ is of the form
   \begin{align*} 
      \phi =\sum_{n\in \naturals}^{\text{finite}} f_{-n} U^{*n} 
              + f_0 + \sum_{n\in \naturals}^{\text{finite}} f_n U^n
   \end{align*}  
   where $f_n\in \mathcal{\tilde{F}}(X)$, 
   $U=p_1\delta_1$ and $U^*=p_{-1}\delta_{-1}$.
   The elements $U$ and $U^*$ fulfil 
   \begin{align*}   
      U U^* = p_1=p_{\text{ran}(\alpha)}, \qquad
      U^* U = p_{-1} = p_{\text{dom}(\alpha)}, \\
      U f = \paraa{(p_{-1} f)\circ \alpha^{-1}} U, \qquad
      U^* f = \paraa{(p_{1} f)\circ \alpha} U^*
   \end{align*}  
   for $f\in \mathcal{\tilde{F}}(X)$.
\end{proposition}
\begin{proof}
  We define $p_n=p_{\alpha_n}$ to be the partial identity of
  $\mathcal{\tilde{F}}(X_n)$ for $n\in\integers$.
  From \ref{single_bij_action} we know that
  $X_n\subset X_{n-1}$ and $X_{-n}\subset X_{-(n-1)}$, thus
  $p_n p_m = p_n$ and $p_{-n} p_{-m} = p_{-n}$ for
  $n>m\in \naturals$.

  Let $n, m\in \naturals$, then
  \begin{align*}
    U_n U_m = p_n U_{n+m} = p_n p_{n+m}\delta_{\alpha_{n+m}}
       =p_n p_{n+m}\delta_{\alpha_{n+m}} =p_{n+m}\delta_{\alpha_{n+m}}
       = U_{n+m}
  \end{align*}
  since $n<n+m$. Analogously, $U_{-n} U_{-m} = U_{-n-m}$. It follows
  by induction that $U_n=U_1^n$ and $U_{-n}=U_{-1}^n$. 
  Furthermore, $U_{-1}=U_1^*=U^*$, and we can replace $U_n$ in
  proposition \ref{pa_cpa_pi} accordingly.
\end{proof}

In the end of this section, we consider the covariant representation
of proposition \ref{point_hs_rep} for the inverse
semigroup $S_\alpha(\integers)$ of proposition \ref{single_bij_sg}.
This covariant representation is a representation of
the crossed product algebra of proposition \ref{single_bij_cpa}
and \ref{single_bij_cpa_pi}.

\begin{lemma} \label{point_hs_pga}
  Let $\alpha$ be a partial group action of the group $G$
  on the set $X$ and let $x_i \in X$ be a finite set of points.
    Then the set $\tilde{X}(\alpha, x_i)$ of definition
    \ref{point_hs} generated by the
    inverse semigroup $S_\alpha(G)$ is
    \begin{align*}
       \tilde{X}(\alpha, x_i)= \{\alpha_g(x_i) | g\in G, i\in I, \text{ s. t. } x_i\in X_{g^*}\}
    \end{align*}
\end{lemma} 
\begin{proof}
   Let us denote 
   \begin{align*}
      \tilde{X}_G= \{\alpha_g(x_i) | g\in G, i\in I, \text{ s. t. } x_i\in X_{g^*}\}
   \end{align*}
  By definition the set $\tilde{X}(\alpha, x_i)$, 
  which we abbreviate with $\tilde{X}_S$,
  is generated by all elements of the
  inverse semigroup $S_\alpha(G)$
  \begin{align*}
      \tilde{X}_S = \{\alpha_s(x_i) | s\in S_\alpha(G), i\in I, \text{ s. t. } x_i\in X_{s^*}\}
  \end{align*}
  The elements of $S_\alpha(G)$ all can be written
  as $s = \epsilon_{g_1} \cdots \epsilon_{g_n} \alpha_g$
  (see proposition \ref{pga_sg_elem} ).
  The projectors $\epsilon_{g_1}$ are identities on subsets of $X$ and
  thefore $\epsilon_{g_i}\alpha_g(x_i)=\alpha_g(x_i)$, when $\alpha_g(x_i)$
  is in the domain of $\epsilon_{g_i}$, or $\epsilon_{g_i}\alpha_g(x_i)=\emptyset$,
  when not. It follows that $\tilde{X}_S \subset \tilde{X}_G$.
  
  Since the $\alpha_g$ for $g\in G$ are elements of  $S_\alpha(G)$, it
  is clear that $\tilde{X}_G \subset \tilde{X}_S$ and therefore $\tilde{X}_G=\tilde{X}_S$.
\end{proof}

\begin{proposition}  \label{point_hs_rep_pg}
   Let $\alpha$ be a partial group action of the group $G$ on the set $X$,
   let $x_i \in X, i \in I$ with $I$ a finite index set different points in $X$, 
   let $\tilde{X}(\alpha, x_i)$ be the generated point set such as in 
   proposition \ref{point_hs_pga} and $\mathcal{H}(\alpha, x_i)$
   the Hilbert space such as in definition \ref{point_hs}.

   Let $\mathcal{\tilde{F}}(X)$ be
    one of the function algebras
  $\mathcal{F}(X)$, $C(X)$ or $C_0(X)$
  (remark \ref{com_fun_alg}).
   Then for $s=\alpha_{g_1}\cdots \alpha_{g_n}\in S_\alpha(G)$,
   $y \in \tilde{X}(\alpha, x_i)$ and $f \in \tilde{\mathcal{F}}(X)$
   \begin{align*}
      V_s \ket{y} = \begin{cases}
  		              \ket{\alpha_{g_1}\cdots \alpha_{g_n}(y)}, \text{ if } y\in X_{s^*} \\
                           0, \text{ otherwise}
                        \end{cases} ,  \qquad
      \pi(f) \ket{y} = f(y) \ket{y}
   \end{align*}
   is a covariant representation on $\mathcal{H}(\alpha_s, x_i)$
   with respect to the induced action $\hat{\alpha}$
   of $S_\alpha(G)$ on the algebra $\tilde{\mathcal{F}}(X)$.
   
   The action is generated by the operators $V_g=V_{\alpha_g}$ 
   for $g\in G$ via $V_s= V_{g_1}\cdots V_{g_n}$
   In particular, the $V_g$ are partial isometries
   with $V_{g^{-1}}=V_g^*$, $V_{gh}V_{h^{-1}}=V_{g} V_{h} V_{h^{-1}}$
   and $V_{g^{-1}}V_{gh}=V_{g^{-1}} V_{g} V_{h}$
   for all $g,h \in G$.
\end{proposition} 
\begin{proof}
  Let $s = \alpha_{g_1}\cdots \alpha_{g_n} \in S_\alpha(G)$.
  According to proposition \ref{point_hs_rep}, the
  operators $V_s$ form a representation of partial
  isometries and together with the algebra representation $\pi$ a
  covariant representation. By definition
  \begin{align*}
      V_s \ket{y} = \ket{\alpha_s(y)} = \ket{\alpha_{g_1}\cdots \alpha_{g_n}(y)}
       = V_{g_1} \cdots V_{g_n} \ket{y} 
   \end{align*}
   for $y\in X_{s^*}$. 
   Furthermore,
  \begin{align*}
      V_{g^{-1}}= V_{\alpha_{g^{-1}}} = V_{\alpha^*_g}
       = \paraa{V_{\alpha_g}}^* = V_g^*
   \end{align*}
   and   
  \begin{align*}
      V_{gh}V_{h^{-1}}=V_{\alpha_{gh}} V_{\alpha_{h^{-1}} } 
        =V_{\alpha_{gh} \alpha_{h^{-1}} } 
        = V_{\alpha_g\alpha_h \alpha_{h^{-1}} } 
        = V_{\alpha_g}V_{\alpha_h} V_{\alpha_{h^{-1}} }
        = V_g V_h V_{h^{-1}} 
   \end{align*}
   by proposition \ref{pga_prop}. The other equation is proven
   equivalently.
\end{proof}

\bibliographystyle{alpha}
\bibliography{references}  

\end{document}